\documentclass[11pt,reqno]{amsart}
\usepackage[utf8]{inputenc}
\usepackage{pgfplots}
\usepackage{fancyhdr}
\usepackage{tikz}
\usepackage{systeme}
\usepackage{bbm}
\usepackage{graphicx}
\usepackage[rightcaption]{sidecap}
\graphicspath{ {images/} }
\usepackage[utf8]{inputenc}
\usepackage{amsmath,dsfont,amssymb,mathrsfs,amsthm,hyperref,color}
\usepackage{graphicx}
\usepackage[english]{babel}
\usepackage{bbm}
 \usepackage{esint}
\usepackage{pst-node}
\usepackage{tikz-cd} 
\usepackage{enumitem}

\usepackage{tikz}
\usetikzlibrary{arrows.meta,calc}

\newtheorem{theorem}{Theorem}[section]
\newtheorem{corollary}[theorem]{Corollary}
\newtheorem{lemma}[theorem]{Lemma}

\theoremstyle{definition}
\newtheorem{definition}[theorem]{Definition}

\theoremstyle{remark}    
\newtheorem{remark}[theorem]{Remark}

\title[Path-Space OT with Interactions: Kinetic Equations and Eikonal Structure] {Path-Space Optimal Transport with Interactions: Kinetic Equations and Eikonal Structure}

     \author{Rene Cabrera}
    \thanks{}
     \address{Department of Mathematics\\
The University of Texas at Austin\\
2515 Speedway Stop C1200\\
Austin, Texas 78712-1202}
\email{rene.cabrera@math.utexas.edu}
\begin{document}
 \maketitle 
\begin{abstract}
We study a dynamical optimal transport problem on path-space with kinetic cost and nonlocal interaction. For a Gaussian interaction kernel, cyclical monotonicity and the regularity of the effective endpoint cost yield a conservative vector field associated with the initial momentum of the optimal trajectories. Separately, the Gaussian interaction potential $W_{\pi_0}$ satisfies a uniform gradient estimate and, under a quantitative condition on the interaction strength, is a classical and hence viscosity subsolution of an eikonal equation. We also derive the Euler--Lagrange dynamics directly from path-space optimality by means of endpoint-preserving perturbations of the optimal path measure. The associated phase-space marginals satisfy a measure-valued Liouville-type, or nonlocal kinetic, equation driven by the smooth Gaussian self-consistent force. Conversely, a superposition principle lifts suitable measure-valued solutions of the phase-space continuity equation to measures on phase-space trajectories. Thus path-space optimal transport provides a variational connection between endpoint geometry, eikonal behavior of the Gaussian interaction potential, and kinetic mean-field dynamics.
\end{abstract}
 \tableofcontents
\section{Introduction and Main results}
\label{sec}

The purpose of this paper is to connect an interacting optimal transport
problem on path space with three closely related structures generated by the
same optimal path measure: an endpoint geometric structure, an eikonal
subsolution associated with the Gaussian interaction potential, and a measure-valued kinetic evolution in phase-space. We specialize the interaction to
the smooth Gaussian kernel. Furthermore, one of the main themes of this note is also to show how the theory of optimal transport as was developed by Brenier \cite{Brenier1991}, Gangbo-McCann \cite{gangbo1996geometry}, McCann \cite{1077246085}, Villani \cite{villani2003topics}-\cite{villani2008optimal}, and Evans \cite{evans1997monge}, just to name a few, provides a connection of optimal transport under the path-space with congestion to Kinetic theory---Vlasov type equation. 

The central focus of the paper is that the
variational structure of the path-space problem contains both an endpoint
geometry and a phase-space dynamics. On the endpoint side, an Evans-type
cyclical-monotonicity argument \cite{evans1997monge} produces a conservative vector field and an
eikonal-type structure. On the dynamical side, a first variation of the
path-space functional yields the Euler--Lagrange equation satisfied by almost
every optimal trajectory, and the corresponding position-velocity
push-forward gives a measure-valued nonlocal kinetic equation. Thus these two
PDE structures, although different in nature, have a common origin in the
same optimal probability measure on paths.

We collect the main results of this note. Let
$\pi_0\in\Pi_{\mathrm{path}}(\mu,\nu)$ be an optimal path measure for the
problem \eqref{eqn: OT paths}, and let
$\mu_t=(e_t)_{\#}\pi_0$. For the Gaussian interaction
$\mathcal K_G(z)=\lambda e^{-\beta|z|^2}$, define the induced force by
\begin{align}\label{eq:gaussian-self-consistent-force}
K_{\pi_0}(t,x)
:=
-\nabla_xV(t,x)
-
2\lambda\beta
\int_{\mathbb R^d}
(x-y)e^{-\beta|x-y|^2}\,d\mu_t(y).
\end{align}

One of the main results connecting optimal transport with path space and interaction to a Liuoville type PDE is the following. (See Theorem \ref{thm:path-implies-liouville} for the full statement.)
\begin{theorem}\label{thm:main-liouville}
    Let $\pi_0$ be an admissible path-space measure minimizng \eqref{eqn: OT paths} and 
concentrated on $W^{2,1}$ solutions of
$\ddot\gamma=K_{\pi_0}(t,\gamma)$,
and assume the velocity and force are integrable
with respect to $dt\otimes\pi_0$.
Then $\phi_t=(E_t)_\#\pi_0$ is weakly continuous
and satisfies
\[
\partial_t\phi_t
+\operatorname{div}_x(v\phi_t)
+\operatorname{div}_v
\bigl(K_{\pi_0}(t,x)\phi_t\bigr)=0
\]
in the sense of distributions, i.e., $\mathcal D'((0,1)\times\mathbb R^{2d})$.
\end{theorem}

\begin{theorem}
\label{thm:main-kinetic}
 Assume that the hypotheses of
Theorem~\ref{thm:path-implies-liouville} hold (i.e., Theorem \ref{thm:main-liouville}).
Suppose, in addition, that
\[
K_{\pi_0}(t,x)
=
-\nabla_x V(t,x)
+
\nabla_x W_{\pi_0}(t,x),
\]
where $W_{\pi_0}$ is generated by the Gaussian
interaction kernel
\[
\mathcal K_G(z)=\lambda e^{-\beta|z|^2},
\qquad \lambda,\beta>0,
\]
and $\nabla_{x}W_{\pi_{0}}$ is given by \eqref{eq:grad-W-gaussian}, while $\nabla_x V$ is bounded and uniformly
Lipschitz in $x$, independently of $t$. In particular, the optimal trajectories have
sufficient regularity to define the phase-space measures
$\phi_t=(e_t,\dot e_t)_{\#}\pi_0$. These measures satisfy
\begin{align}\label{PDE Vlasov}
\partial_t\phi_t
+
\operatorname{div}_x(v\phi_t)
+
\operatorname{div}_v
\bigl(K_{\pi_0}(t,x)\phi_t\bigr)
=0
\end{align}
in the sense of distributions.

If, in addition, the initial phase-space measure is absolutely continuous with respect to Lebesgue,
$\phi_0=f_0\,\mathcal L^{2d}$, then absolute continuity is propagated by the
characteristic flow. Hence
$\phi_t=f_t\,\mathcal L^{2d}$, where $f_{t}=f_{0}\circ\Psi_{t}^{-1}$, and the density satisfies
\begin{align}\label{PDE Liouville}
\partial_t f
+
v\cdot\nabla_x f
+
K_{\pi_0}(t,x)\cdot\nabla_v f
=0
\end{align}
in the distributional sense. Moreover, we have conservation of momentum, i.e., initial condition
$f_0\in L^p(\mathbb R^{2d})$, $1\leq p\leq\infty$, then
\[
\|f_t\|_{L^p(\mathbb R^{2d})}
=
\|f_0\|_{L^p(\mathbb R^{2d})}
\]
for every $t\in[0,1]$.
\end{theorem} See Theorem \ref{thm:gaussian-density-propagation}.

Although both results above describe the same phase-space dynamics and have similarities between them, they are not technically identical. Indeed, the Theorems \ref{thm:main-liouville} and \ref{thm:main-kinetic} make different assumptions and establish different conclusions. To wit, Theorem \ref{thm:main-liouville} states that if $\pi_{0}$ is concentrated on the trajectories satisfying $\ddot{\gamma}=K_{\pi_{0}}(t,\gamma(t))$, then defining $\phi_{t}:=(E_{t})_{\#}\pi_{0}$ it satisfies \eqref{PDE Vlasov}. 

While Theorem \ref{thm:main-kinetic} establishes additional hypothesis. Indeed it assumes initial data to the effect of being absolutely continuous with respect to Lebesgue measure, namely a density function $\phi_0 = f_0\,\mathcal L^{2d},$
and deduces the assertion $\phi_t = f_t\,\mathcal L^{2d},$ in general for all $t$,
where $f_t = f_0 \circ \Psi_t^{-1},$ while satisfying \eqref{PDE Liouville}. It also establishes conservation of energy
$\|f_t\|_{L^p} = \|f_0\|_{L^p}.$
These are additional properties of solutions to the same Liouville equation \eqref{PDE Liouville}.

A partial converse to Theorem \ref{thm:main-kinetic} is provided. To this end, the kinetic description can be lifted back to path space. In other words, a kinetic solution admits a superposition by characteristic curves, and projecting these characteristics recovers a path-space measure whose phase-space marginals are precisely $\phi_t$.

\begin{theorem}\label{thm:phase-space-superposition-intro} Let $\{\phi_t\}_{t\in[0,1]}$ be a family of probability measures in $\mathbb{R}^{2d}$ satisfying \[ \partial_t\phi_t+ \operatorname{div}_{(x,v)} \bigl((v,K(t,x))\phi_t\bigr)=0. \] Under suitable integrability assumptions, there exists a probability measure $\eta_{\textbf{ph}}$ concentrated on characteristics, for position $X(t,x)$ and velocity $V(t,x)$ of the dynamics, \[ \dot X=V,\qquad \dot V=K(t,X), \] whose time marginals are $\phi_t$. Moreover, projecting $\eta_{\textbf{ph}}\in\mathcal{P}(AC([0,1];\mathbb{R}^{2d}))$ onto path-space yields a measure $\pi$ satisfying 
\begin{align}\label{eq:phase-marginal-from-path-measure}
(e_t,\dot e_t)_{\#}\pi=\phi_t, \qquad t\in[0,1].
\end{align}
\end{theorem}
(See Theorem \ref{thm:phase-space-superposition} for a partial converse).

The third main result concerns the eikonal structure associated with the endpoint map generated by the solution of the interacting path-space optimal transport problem. In particular, the cyclical geometric argument, stemming from a solution to the optimal transport problem, introduced by Evans \cite{evans1997monge} in the derivation of the eikonal equation is adapted to our setting. In contrast with Theorem~\ref{thm:main-kinetic}, which is obtained from variations of the optimal trajectories and describes their phase-space dynamics, the following result is obtained from the endpoint geometry of the same optimal path measure. 
\begin{theorem} \label{thm:main-eikonal} Let $\pi_0$ be a solution of the interacting path-space optimal transport problem \eqref{eqn: OT paths}, where the interaction is induced by the Gaussian kernel \begin{equation}
\label{eq:gaussian-kernel-dynamics}
\mathcal K_G(z)=\lambda e^{-\beta|z|^2},
\qquad \lambda,\beta>0.
\end{equation}
Assume the hypotheses on $\mu$, $\nu$, the path cost $c$, the external potential $V$, and the Gaussian interaction
\begin{align}
\label{self consistent Gaussian potential}
\begin{split}
W_{\pi_0}(t,x)
=
(\mathcal K_G*\mu_t)(x)
=
\lambda
\int_{\mathbb R^d}
e^{-\beta|x-y|^2}\,d\mu_t(y).
\end{split}
\end{align}
Let \[ \pi_0=(x\mapsto\gamma_x)_{\#}\mu \] be the corresponding deterministic representation of the optimal path measure, and let \[ T^*(x):=\gamma_x(1) \] denote the induced optimal endpoint map. Then the scalar potential \eqref{self consistent Gaussian potential} (associated with the endpoint transport field generated by $T^*$ )is a viscosity subsolution of the eikonal equation 
\begin{equation}\label{eqn: eikonal} |\nabla u(x)|=1 \end{equation} on the domain specified below, in the viscosity sense. 
\end{theorem} 

The previous assertions have introduced three different measures that will be instrumental in this manuscript. Firstly, the end-point coupling $\eta_{0}$ is an optimal measure of the end-points of a cost, $c(x,y)$. That is, $\eta_{0}$ optimizes  \eqref{end-point cost} from the initial point $x$ to final point $y$. In particular, the deterministic mapping, i.e., $y=T^{*}(x)$ is such that $\eta_{0}=(\text{id}, T^{*})_{\#}\mu$, and $T^{*}$ solves \eqref{eq:OT}, and hence \eqref{optimal solution} \cite{cabrera2021optimal}. In summary, $\eta_{0}$ is the optimal endpoint coupling. 

The novel phase-space trajectory measure, $\eta_{\textbf{ph}}$ is introduced by the super position principle. The measure $\eta_{\textbf{ph}}$ lies in $\mathcal{P}(AC([0, 1];\mathbb{R}^{2d}))$. Elements in this space are entire phase-space trajectory determined by the dynmics of a particles position and velocity for some time $t$. That is, the elements are determined by trajectories, $Z(t)=(X(t), V(t))$. Thus $\eta_{\mathrm{ph}}$ does not describe where particles are merely at one time. It gives a probability distribution over complete position-velocity histories. 
For example, one trajectory in its support looks like $ t\longmapsto Z(t) = (X(t),V(t)) $
with
$\dot X(t)=V(t), \qquad \dot V(t)=K(t,X(t))$.

At each time $t$, one evaluates these trajectories using
$ \operatorname{ev}_t(X,V) = (X(t),V(t))$. The superposition principle says
$ (\operatorname{ev}_t)_\#\eta_{\mathrm{ph}} = \phi_t. $ Therefore, $\eta_{\mathrm{ph}}$ is probability measure on entire position--velocity trajectories.

On the other hand, $\pi$ is the position-path measure. Starting with $\eta_{\mathrm{ph}}$, forget the explicit velocity component. Define $P(X,V)=X$. Then
$ \pi:=P_\#\eta_{\mathrm{ph}}$. Thus the map is 
$ (X(\cdot),V(\cdot)) \overset{P}{\longmapsto} X(\cdot). $

While $\eta_{\mathrm{ph}}$ is a measure on the space of absolutely continuous paths
$ AC([0,1];\mathbb R^{2d}), $
the resulting $\pi$ is a measure on position paths. After establishing regularity, see section \ref{subsec:gaussian-effective-action},  $X\in C^1([0,1];\mathbb R^d)$; so on can regard
$\pi\in \mathcal P(C^1([0,1];\mathbb R^d))$.
Hence, $\pi$ is the probability measure on position trajectories. For all intents and purposes, $\pi$ reconstructs the converse theorem (Theorem \ref{thm:phase-space-superposition}) which is not automatically the optimal $\pi_{0}$. It is a path measure realizing the given \((\phi_t)\). 

Finally, fix time \(t\).
We no longer want the entire trajectory. The question is such that:
What is the joint distribution of position and velocity at this particular time? That is $ \phi_t\in\mathcal P(\mathbb R^{2d}). $ Its variables are $(x,v)$. There are two equivalent ways to obtain this in Theorem \ref{thm:phase-space-superposition}, the converse construction.
From the phase-space trajectory measure, we have
$\phi_t = (\operatorname{ev}_t)_\#\eta_{\mathrm{ph}}$. From the position-path measure,we have $\phi_t = (e_t,\dot e_t)_\#\pi $. Why are these equal? Because
we have $P(X,V)=X$ and along these characteristics we have $ \dot X=V. $
Therefore,
$ (e_t,\dot e_t)(P(X,V)) = (X(t),\dot X(t)) = (X(t),V(t)) = \operatorname{ev}_t(X,V). $

Thus $(e_t,\dot e_t)\circ P = \operatorname{ev}_t $ and consequently $(e_t,\dot e_t)_\#\pi = (\operatorname{ev}_t)_\#\eta_{\mathrm{ph}} = \phi_t $.
In conclusion $\phi_{t}$ is the position–velocity distribution at time $t$.

\subsection{Motivation}
Theorem~\ref{thm:main-eikonal} is the endpoint PDE consequence of the path-space variational problem. Its proof is based on an adaptation of the geometric argument of Evans \cite[Section 8.1]{evans1997monge}. The essential step is to pass from optimality of $\pi_0$ on path space to an effective endpoint optimality principle. Cyclical monotonicity may then be applied to the resulting endpoint problem along closed curves. This produces the zero-circulation identity for the normalized endpoint displacement field \[ \xi(x) = \frac{T^*(x)-x}{|T^*(x)-x|}. \] The Evans argument is therefore an ingredient in the proof of Theorem~\ref{thm:main-eikonal}, rather than the main theorem itself. The two main theorems describe different consequences of the same optimizer $\pi_0$. Theorem~\ref{thm:main-kinetic} uses the dynamical information contained in the optimal trajectories to obtain the measure-valued kinetic equation \eqref{PDE Vlasov}. While Theorem \ref{thm:phase-space-superposition-intro} says that one can go from the kinetic description back to a path-space description. But, what the theorem does not say is that the reconstructed $\pi$ is an optimal path measure for \eqref{eqn: OT paths}. It merely says that a kinetic solution can be represented by a measure on paths following the characteristic dynamics. Therefore it is a converse to the representation/dynamical passage in the kinetic theorem, Theorem \ref{thm:main-kinetic}, not a converse to optimality itself. Theorem~\ref{thm:main-eikonal} uses the endpoint information contained in those trajectories to obtain a viscosity subsolution of the eikonal equation \eqref{eqn: eikonal}. The central thesis of the paper is precisely that these two PDE structures have the same variational origin: the solution of the interacting optimal transport problem on path space.

Theorems \ref{thm:main-kinetic}, \ref{thm:phase-space-superposition-intro} and \ref{thm:main-eikonal} summarize the
principal contribution of the paper. The first describes the dynamics of the
optimal trajectories in phase-space giving rise to solutions of Lioville-type equations, while the second describes offers a partial converse to Theorem \ref{thm:gaussian-density-propagation}, and the third result describes the geometry
of their endpoints thusly producing a subsolution of an eikonal type equation as a consequence. The assertions originate from the same path-space optimizer
$\pi_0$. The kinetic result is obtained from variations of the entire optimal
path measure, while the endpoint result is obtained from the effective
endpoint cost and cyclical monotonicity. This common variational origin is
the main structural connection developed in the paper.

The endpoint argument developed in this paper is motivated by the geometric approach of Evans \cite[Section 8.1]{evans1997monge}, in which cyclical monotonicity is used to obtain a zero-circulation identity along closed curves and, consequently, a potential associated with the optimal transport map. We adapt this idea to the effective endpoint problem induced by the interacting path-space optimal transport problem. In our setting, the Evans argument is an ingredient in proving that the potential associated with the optimal endpoint map satisfies the viscosity subsolution property for the eikonal equation.

We reiterate the importance of distinguishing different measures generated by the
optimal path measure, $\pi_{0}$ satisfying \eqref{eqn: OT paths}. The probability measure $\pi_0$ is defined on path
space, $\Omega$, and is a solution of the interacting path-space optimal transport
problem \eqref{eqn: OT paths}. Its spatial marginal at time $t$ is $\mu_t:=(e_t)_{\#}\pi_0$,
which is obtained by projecting the path measure onto Euclidean space through
the evaluation map $e_t(\gamma)=\gamma(t)$. In particular,
$\mu_0=\mu$ and $\mu_1=\nu$. The corresponding endpoint coupling is $\eta_{0}:=(e_0,e_1)_{\#}\pi_0.$ Under the deterministic representation of the optimal path measure,
$\pi_0=(x\mapsto\gamma_x)_{\#}\mu$, this endpoint coupling takes the form
\[
\eta_{0}=(\operatorname{id},T^*)_{\#}\mu,
\qquad
T^*(x)=\gamma_x(1).
\]
It is this endpoint structure, together with the effective endpoint cost \eqref{end-point cost}
induced by the path-space problem \eqref{effective cost}, that is used in the cyclical-monotonicity
argument leading to the eikonal result.

A different projection of the same optimal path measure gives the kinetic
description. Once the optimal trajectories have sufficient regularity, we
define the phase-space marginals by $(e_{t},\dot{e}_{t})_{\#}\pi=\phi_{t}$. Namely, the optimal path plan $\pi_0$ satisfies \eqref{eq:phase-marginal-from-path-measure}.

Thus $\pi_0$ itself is not asserted to solve a Liouville equation. Rather,
the family of phase-space measures $\{\phi_t\}_{t\in[0,1]}$ generated by
$\pi_0$ satisfies the measure-valued Liouville-type equation \eqref{PDE Liouville} derived below.
In this way, the endpoint and kinetic results arise from two different
projections of the same optimal probability measure on paths.

\subsection{Related work}
A related variational connection with kinetic equations was recently investigated by Gladbach and Kepka~\cite{GladbachKepka2025}, who study variational interacting particle systems and show that critical points of the corresponding action give rise to weak solutions of a Vlasov equation. Their work also addresses relaxation, particle approximations, and dynamic interacting-particle optimal transport. The approach developed here has a different starting point and emphasis: our kinetic description stems directly from the path-space optimal transport problem with interaction and congestion. More precisely, optimality of the path measure \(\pi_0\) yields the nonlocal Euler--Lagrange dynamics \[ \ddot{\gamma}(t) = K_{\pi_0}\bigl(t,\gamma(t)\bigr) \] for \(\pi_0\)-almost every path, and the associated phase-space measures are then generated directly from the optimal path plan by \[ \phi_t = (e_t,\dot e_t)_{\#}\pi_0. \] The resulting kinetic equation is therefore obtained as the phase-space evolution induced by the optimal trajectories of the interacting path-space transport problem.
\subsection{Background on optimal transport}\label{back ground}
Next, we briefly recall the optimal transport framework needed for the application of Evans' heuristic construction. Given probability measures $\mu$ on $\mathcal X$ and $\nu$ on $\mathcal Y$, the Kantorovich optimal transport problem is 
\begin{equation} \label{eq:OT}
W_c(\mu,\nu) = \inf_{\pi\in\Pi(\mu,\nu)} \int_{\mathcal X\times\mathcal Y}c(x,y)\,d\pi(x,y), 
\end{equation}
where $\Pi(\mu,\nu)$ denotes the collection of probability measures having marginals $\mu$ and $\nu$.
For $c(x,y)=\|x-y\|^p$, the $p$-th root of the
optimal value gives the $p$-Wasserstein distance.

Although the dual problem is not the principal tool used in the present
paper, it is fundamental to the geometric theory of optimal transport.
Kantorovich duality takes the form
\begin{equation}
\label{eq:duality}
W_c(\mu,\nu)
=
\sup_{\substack{\varphi\in C(\mathcal X),\,\psi\in C(\mathcal Y)\\
\varphi(x)+\psi(y)\leq c(x,y)}}
\left\{
\int_{\mathcal X}\varphi\,d\mu
+
\int_{\mathcal Y}\psi\,d\nu
\right\}.
\end{equation}
The dual formulation, together with cyclical monotonicity, provides a
fundamental characterization of optimal transport plans; see
\cite{ambrosio2013user,villani2003topics,villani2008optimal}. These ideas
underlie the relation between optimal transport and gradient structures that
motivates the endpoint part of this paper. See Lemma \ref{dual problem}.

The classical Monge problem asks instead for a deterministic transport map.
Given $\mu,\nu\in\mathcal P(\mathbb R^d)$, one considers
\begin{equation}\label{eqn: Monge}
\mathcal M(T)
=
\int_{\mathbb R^d}c(x,T(x))\,d\mu(x).
\end{equation}
Every admissible transport map $T$ induces the Kantorovich plan
$(\operatorname{Id}\times T)_{\#}\mu$, which is concentrated on the graph of
$T$. Conversely, a transport plan concentrated on the graph of a measurable
map is induced by such a map. For the quadratic cost, Brenier
\cite{Brenier1991} proved, under the standard absolute-continuity hypotheses,
that the optimal map is the gradient of a convex function. Related
characterizations were developed by McCann \cite{1077246085}, Knott and
Smith \cite{Knott1984,Smith1992}, and others. We refer to
\cite{villani2003topics,villani2008optimal,ambrosio2013user} for the broader
theory.

The path-space problem studied here originates in the optimal transport
framework with paths and congestion developed by the author in
\cite{cabrera2021optimal,Cabrera2022}. The corresponding minimization problem
is
\begin{align}\label{eqn: OT paths}
\begin{split}
\int_{\Omega}c(\gamma)d\pi_{0}(\gamma)&+\frac{1}{2}\iint_{\Omega\times\Omega}\mathcal{U}(\gamma,\sigma)d\pi_{0}(\gamma)d\pi_{0}(\sigma)\\
&=\inf_{\pi\in\Pi_{\mathrm{path}}(\mu,\nu)}
\int_{\Omega}c(\gamma)\,d\pi(\gamma)
+
\frac{1}{2}
\iint_{\Omega\times\Omega}
\mathcal U(\gamma,\sigma)\,
d\pi(\gamma)d\pi(\sigma).
\end{split}
\end{align}
Here, let $I:=[0, 1]$,
$\Omega=\{\gamma:I\to\mathbb R^d:\gamma\text{ is absolutely continuous}\}$
is the path space and $\Pi_{\mathrm{path}}(\mu,\nu)$ denotes the set of
probability measures on $\Omega$ satisfying
$(e_0)_{\#}\pi=\mu$ and $(e_1)_{\#}\pi=\nu$, where
$e_t(\gamma)=\gamma(t)$ is the evaluation map. In particular, for each $t\in[0,1]$, let
\[
e_t:\Omega\to\mathbb R^d,
\qquad
e_t(\gamma)=\gamma(t),
\]
and
\[
\dot e_t:T\Omega\to\mathbb R^d,
\qquad
\dot e_t(\gamma)=\dot\gamma(t),
\]
denote the position and velocity evaluation maps,
respectively. Furthermore, we make the following identifications. For the path $\gamma: [0, 1]\to X$, we identify its position $\gamma(t)=x$ and velocity $\dot{\gamma}(t)=v(t,\gamma(t))\in \mathbb{R}^{d}$. We define the phase-space evaluation
map by
\[
E_t=(e_t,\dot e_t):
\Omega\times T\Omega\to\mathbb R^{2d}.
\]

It makes sense to talk about the derivative of the evaluation map $e_{t}$ for all $t$. Which aligns with the derivative of $\gamma$. To this end, we have the following by definition
\begin{align*}
    \dot{e}_{t}=\lim_{h\to 0}\frac{e_{t+h}-e_{t}}{h}.
\end{align*}
Then evaluating an arbitrary path $\gamma$, we have
\begin{align*}
    (\dot{e}_{t})(\gamma)=\left(\lim_{h\to 0}\frac{e_{t+h}-e_{t}}{h}\right)(\gamma)&=\lim_{h\to 0}\frac{e_{t+h}-e_{t}}{h}(\gamma)=\lim_{h\to 0}\frac{e_{t+h}(\gamma)-e_{t}(\gamma)}{h}\\
    &=\lim_{h\to0}\frac{\gamma(t+h)-\gamma(t)}{h}=:\dot{\gamma}(t).
\end{align*}
Note that the mapping $E_{t}$ in \eqref{eq:position-velocity-evaluation} in Definition \ref{def:phase-space-marginal} has domain $\Omega^{1}:=C^{1}([0, 1];\mathbb{R}^{d})$. We make this change in domain as we are working with two different levels of regularity. Indeed, our path-space, $\Omega$, is essentially the Hilbert space, $H^{1}([0, 1];\mathbb{R}^{d})$. THis is a natural space for our kinetic cost \eqref{KE} as finite kinetic energy requires only $\dot{\gamma}\in L^{2}$. 

While the change in domain to $\Omega^1$ is due to the phase-space formulation. In other words, the Euler-Lagrange equations that we apply and the boundedness of the force $K_{\pi_{0}}$ in \eqref{eq:gaussian-EL} establishes $\gamma \in W^{2,\infty}([0,1];\mathbb{R}^d) \subset C^1([0,1];\mathbb{R}^d) $ for $\pi_0$-almost every $\gamma$. Therefore, $\pi_{0}\left(\Omega^1\right)=1$. This additional regularity allows us to define $E_{t}$ as above or in \eqref{eq:position-velocity-evaluation} in Definition \ref{def:phase-space-marginal}. See Section \ref{sec:Phase-Space} for details.

The corresponding phase-space marginal is given by \eqref{eq:phase-marginal-from-path-measure}. In particular,  we have the following equality, $\phi_{t}=(E_{t})_{\#}\pi=(e(\gamma),\dot{e}(\gamma))_{\#}\pi=(\gamma(t),\dot{\gamma}(t))_{\#}\pi$.
Let $X\subset\mathbb{R}^{d}$ be a simply connected domain.

The path cost $c:\Omega\to X$ is given by the Lagrangian
\begin{align}\label{kinetic energy}
c(\gamma)
=
\begin{cases}
\displaystyle
\int_0^1
\left(
\frac12|\dot\gamma(t)|^2-V(t, \gamma(t))
\right)\,dt,
& \dot\gamma\in L^2(dt),\\[2mm]
+\infty,
& \text{otherwise},
\end{cases}
\end{align}
under the assumptions on $V$. The interaction kernels considered
in the underlying path-space theory are even and positive definite,
\begin{align}\label{kernel}
\mathcal K:\mathbb R^d\to\mathbb R,
\end{align}
and induce the path interaction
\begin{align}\label{interaction functional}
\mathcal U(\gamma,\sigma)
=
\int_0^1
\mathcal K(\gamma(t)-\sigma(t))\,dt.
\end{align}
Throughout the present paper we specialize this interaction to
$\mathcal K=\mathcal K_G$, where
$\mathcal K_G(z)=\lambda e^{-\beta|z|^2}$. 
 
The existence and structural theory for \eqref{eqn: OT paths} was developed
in \cite{cabrera2021optimal,Cabrera2022}. In particular, under the hypotheses
used there, the problem admits an optimal path measure, and under the
appropriate compactness and absolute-continuity assumptions this optimizer
has a deterministic representation
$\pi_0=(\Gamma)_{\#}\mu$, where
$\Gamma(x)=\gamma_{x,T^*(x)}$ (this is equivalent to $x\mapsto\gamma_{x}$) and $T^*_{\#}\mu=\nu$. These results provide
the starting point for the analysis carried out here.

The path-space formulation raises several natural questions that motivate
the present paper. First, if $\pi_0$ is an optimal path measure, what
evolution equation is satisfied by the position--velocity distribution
\begin{align*}
\phi_t=(e_t,\dot e_t)_{\#}\pi_0?
\end{align*}
Second, when does this measure admit a phase-space density, and what
regularity of that density is propagated by the resulting dynamics? Third,
starting instead from a family of distributions
$\{\phi_t\}_{t\in[0,1]}\subset\mathcal P(\mathbb R^{2d})$ satisfying the
appropriate continuity equation, when can it be represented by a probability
measure concentrated on characteristic trajectories? Finally, what
additional geometric information about the endpoint map can be extracted
from the same path-space optimizer?

The results of this paper answer these questions in the Gaussian setting.
The first variation of the path-space problem gives the characteristic
dynamics and hence the measure-valued kinetic equation \eqref{PDE Vlasov}.
When an initial phase-space density exists, the Gaussian characteristic flow
propagates that density and yields \eqref{PDE Liouville}. The superposition
principle provides the converse representation of suitable measure-valued
kinetic solutions by characteristic trajectories. Independently, the
effective endpoint reduction and cyclical-monotonicity argument yield the
conservative endpoint field and the associated eikonal structure (a viscosity subsolution). These
results together realize the central focus of this note stated at the beginning: the
endpoint geometry and the phase-space dynamics are two manifestations of the
same interacting optimal transport problem on path space.

The remainder of the paper is organized as follows. Section \ref{sec. preliminary} reviews cyclical monotonicity and Evans' heuristic argument, which provide the foundation for our investigation of the eikonal subsolution property associated with the endpoint cost induced by the self-consistent Gaussian potential \eqref{self consistent Gaussian potential}. Section \ref{geometry of ot} develops the geometry of paths in optimal transport, introducing congestion-dependent transport costs and constructing an effective pathwise cost and its corresponding endpoint cost. We further investigate the connection between optimal transport and eikonal subsolutions, emphasizing the PDE structure arising from Gaussian interactions.

Section \ref{sec:path-to-phase-space} and its subsequent subsections examine the relationship between optimal path measures and phase-space dynamics. We introduce the effective action and the force field induced by Gaussian interactions, and derive the Euler--Lagrange equations satisfied by optimal trajectories. We then construct the associated phase-space marginals and establish their connection to the kinetic Liouville equation. We further investigate the propagation and regularity of phase-space densities under suitable assumptions. Finally, we establish a converse lifting result through the superposition principle, providing a representation of measure-valued solutions of the kinetic equation in terms of characteristic trajectories.

\section{Preliminary} \label{sec. preliminary}

\subsection{Cyclical monotonicity}
We begin this section with an important definition that characterizes an optimal solution for \eqref{eq:OT}. Then incorporate this theory of \emph{cyclical monotone} sets to paths having initial and final end-points. For the classical definition see Rockafellar \cite{Rockafellar1966} - \cite{rockafellar2015convex} or for a more contemporary one see Villani's, Santambrogio's, or Ambrosio's and Gigli's account on the theory in \cite{villani2003topics}, \cite{santambrogio2015optimal}, \cite{AmbGig2013}.

Consider an admissible plan $\pi$ and let $\gamma_{i}$ be a path from initial point $\gamma_{i}(0)=x_{i}$ to final point $\gamma_{i}(1)=y_{i}$, for $i=1,\dots, N$, and be in the support of $\pi$. Now, let $\{(x_{i}, y_{i})\}_{i=1}^{N}$ be the set of pairs of endpoints of $\gamma_{i}$. The \emph{shift}, as we will call it for the moment, of the final points of paths give rise to new shifted paths $\widetilde{\gamma}$ defined by
  the shifted paths  $\{\widetilde{\gamma}_{i}\}_{i=1}^{N}$ with endpoints set $\{(x_{i}, y_{i+1})\}_{i=1}^{N}$ are such that definition \eqref{cyclical paths} holds and 
 \begin{align*}
     \widetilde{\gamma}_{i}(0)=x_{i}\quad \forall i\quad \text{and}\quad \widetilde{\gamma}_{i}(1)=y_{i+1\;\text{mod}\;N}\quad \forall i
 \end{align*}
with the convention $y_{N+1}=y_{1}$.
\begin{definition}\label{definition of cyclical sets}
We say the set $\{(x_{i},y_{i})\}\subset X$ is \emph{$c$-cyclically monotone} if for all $i=1,\ldots, n$,
 \begin{align*}
     \sum_{i=1}^{n}c(x_{i}, y_{i})\leq \sum_{i=1}^{n}c(x_{i}, y_{\tau(i)}),
 \end{align*}
 for any permutation $\tau$ on $n$ letters. Note that one can think of $\tau(i)=i+1$, with the loop convention, $y_{n+1}=y_{1}$.
\end{definition}
\noindent The cyclical monotonicity definition of paths $\gamma_{i}$ is the following.

\begin{definition}\label{cyclical paths}
    Let $\gamma:[0, 1]\to X$ be a continuously differentiable path such that $\gamma_{i}(0)=x_{i}$ and $\gamma_{i}(1)=y_{i}$ for $i=1,2,\ldots, n$. The set $\{x_{i}, y_{i}\}$ is $c$-cyclical monotone with respect to the paths $\gamma_{i}$ for $i=1,\ldots, n$ if
    \[
    \sum_{i=1}^{n}c(\gamma_{i})\leq\sum_{i=1}^{n}c(\widetilde{\gamma}_{i}),
    \]
    where $\gamma_{i}(0)=x_{i}$ and $\gamma_{i}(1)=y_{i}$, and $\widetilde{\gamma}_{i}(0)=x_{i}$, $\widetilde{\gamma}_{i}(1)=y_{i+1}$. The latter $\widetilde{\gamma}$ measures the shift of the path $\widetilde{\gamma}_{i}$ for each $i$ from $x_{i}$ to $y_{i+1}$. 
\end{definition}
Note that from the above definition, according to the paths $\widetilde{\gamma_{i}}$, the cyclical monotonicity of $c$ asserts that for some sufficiently small $h_{i}$ the \emph{shift} of the final points of paths give rise to new shifted paths $\widetilde{\gamma}$ defined by
  \begin{align}\label{cyclical paths}
      \widetilde{\gamma}_{i}(t)=\gamma_{i}(t)+th_{i}(t),\quad 0\leq t \leq 1,\; h_{i}(t)\neq 0\;\forall i,\quad \text{and}\quad \gamma_{i}(1)+h_{i}(1):=\gamma_{i+1}(1),
  \end{align}
while the initial points of the path remain the same, and; with the convention $\gamma_{N+1}(1)=\gamma_{1}(1)$.  

Rockafellar defined the concept of cyclical monotonicity in his 1966 paper, \emph{Characterization of the subdifferentials of convex functions} \cite{Rockafellar1966}.  Moreover, in the seminal monograph \cite{rockafellar2015convex}, he proved its foundational theorem, establishing it as a cornerstone of convex analysis and, decades later, optimal transport theory. Primarily, he used cyclical monotonicity to solve infinite dimensional linear programs. Furthermore he proved the following theorem in \cite{Rockafellar1966} (often called Rockafellar's Theorem):

\begin{theorem}[Rockafellar, 1966]
A set or multi-valued operator $S$ is cyclically monotone if and only if there exists a convex function $f$ such that $S$ is contained in the subdifferential of $f$ ($S \subseteq \partial f$).
\end{theorem}

This is a major result because it provides a necessary and sufficient condition of when a multi-valued map is cyclical monotone.  In particular, it tells us that the geometric property of cyclical monotonicity is perfectly equivalent to the analytic property of being a subgradient of a convex function. Such a function $f$ is called a \emph{potential} for the cyclically monotone operator $S$.

As a consequence of this foundational work, researchers applied Rockafellar's result to establish a complete theory to characterize optimal transport plans. To wit,  one of the main results in optimal transport theory that fully characterizes optimal plans of \eqref{eq:OT} is the following.
\begin{theorem}[Fundamental theorem of optimal transport \textbf{Theorem 1.13} in \cite{ambrosio2013user}]
\label{thm:fundamental_OT}
Assume that $c : \mathcal{X} \times \mathcal{Y} \rightarrow \mathbb{R}$ is continuous and bounded from below and let $\mu \in \mathcal{P}(\mathcal{X})$, $\nu \in \mathcal{P}(\mathcal{Y})$ be such that
\begin{align*}
c(x,y) \leq a(x) + b(y), 
\end{align*}
for some $a \in L^1(\mu)$, $b \in L^1(\nu)$. Also, let $\pi \in \Pi(\mu, \nu)$. Then the following three are equivalent:
\begin{enumerate}[label=(\roman*)]
\item the plan $\pi$ is optimal;
\item the set $\operatorname{supp}(\pi)$ is $c$-cyclically monotone;
\item there exists a $c$-concave function $\varphi$ such that $\max\{\varphi, 0\} \in L^1(\mu)$ and $\operatorname{supp}(\pi) \subset \partial^{c+}\varphi$.
\end{enumerate}
\end{theorem}

\begin{proof}
    For its proof and definitions therein, we refer the reader to \emph{A user's guide to optimal transport} \cite{ambrosio2013user}.
\end{proof}
In the next section, we will see how Evans used these results, coupled with conservative vector fields, to obtain a vector field which satisfies the eikonal equation.

\section{Geometry of optimal transport}\label{geometry of ot}
The material that will compose this section is drawn from Evans \cite{evans1997monge} and, in particular, from his insightful heuristic computation on the geometry of optimal transport in section 8.1 of his notes \cite{evans1997monge}, which inspired this note. In particular, we reprise the same calculation that Evans proved in order to make this note as self contained as possible.
We look at the distance cost function $c:\mathbb{R}^{d}\times\mathbb{R}^{d}\to \mathbb{R}$ defined by 
\[
c(x,y):=|x-y|.
\]
In addition, we will be concerned with the optimality of \eqref{eqn: Monge}. Brenier developed the techniques in \cite{Brenier1991} that will be discussed in this section.  In order to understand the geometric properties of solutions to $\mathcal{M}(T)$ through the cyclical monotonicity property, we will assume that $T^{*}$ is a minimizer of \eqref{eqn: Monge}. That is, there is  a one-to-one mapping $T^{*}:\mathbb{R}^{d}\to\mathbb{R}^{d}$ that pushes-forward $\mu$ onto $\nu$, among all maps $T$ with the same property that minimizes $\mathcal{M}$, and $T^{*}$ satisfies
\begin{align}\label{optimal solution}
    \mathcal{M}(T^{*})=\inf_{T_{\#}\mu=\nu}\mathcal{M}(T).
\end{align}

Fix a positive integer $N$ and consider distinct points $\{x_{k}\}_{k=1}^{N}$ contained in the support of $\mu$, $\text{spt}(\mu)$. Furthermore, suppose there exist distinct radii $\{r_k\}_{k=1}^{N}$ such that the corresponding small balls $B_k = B(x_k, r_k)$ are distinct, with $r_k > 0$ for all $k$. By adjusting if necessary, suppose that such balls have mass $\varepsilon>0$. In particular,
\begin{align}\label{measure balls}
    \mu(B_{1})=\mu(B_{2})=\cdots=\mu(B_{N})=\varepsilon.
\end{align}
Now since $T^{*}$ is a one-to-one map, then take $y_{k}=T^{*}(x_{k})$ such that $C_{k}=T^{*}(B_{k})$. By assumption, $T^{*}$ pushes-forward $\mu$ onto $\nu$, therefore
\begin{align}\label{measures on C_k}
    \nu(C_{1})=\nu(C_{2})=\cdot=\nu(C_{N})=\varepsilon.
\end{align}

Next, cyclically permuting the images of the balls, $\{B_{k}\}_{k=1}^{N}$,  gives rise to another mapping $T$ with the desired properties among all maps aiming to minimize $\mathcal{M}$. Namely, we build $T$ with the following properties
\begin{align}\label{new T}
\begin{cases}
T(x_{k})=y_{k+1}, & T(B_{k})=C_{k+1}\; (k=1,\ldots, N), \\[4pt]
T=T^{*} & \text{on}\; \text{spt}(\mu)\setminus \bigcup_{k=1}^{N}B_{k}.
\end{cases}
\end{align}

\begin{figure}[h] 
\centering 
\begin{tikzpicture}[>=stealth]\label{pic 1}
\node (x1) at (0,2.5) {$x_1$};
\draw[blue, thick] (0,2.5) circle (0.5);
\node[blue] at (-0.8,2.5) {$B_1$};

\node (x2) at (0,1) {$x_2$};
\draw[blue, thick] (0,1) circle (0.5);
\node[blue] at (-0.8,1) {$B_2$};

\node (x3) at (0,-0.5) {$x_3$};
\draw[blue, thick] (0,-0.5) circle (0.5);
\node[blue] at (-0.8,-0.5) {$B_3$};

\node (xN) at (0,-2.5) {$x_N$};
\draw[blue, thick] (0,-2.5) circle (0.5);
\node[blue] at (-0.8,-2.5) {$B_N$};

\node (y1) at (5,2.5) {$y_1$};
\draw[red, thick] (5,2.5) circle (0.5);
\node[red] at (5.8,2.5) {$C_1$};

\node (y2) at (5,1) {$y_2$};
\draw[red, thick] (5,1) circle (0.5);
\node[red] at (5.8,1) {$C_2$};

\node (y3) at (5,-0.5) {$y_3$};
\draw[red, thick] (5,-0.5) circle (0.5);
\node[red] at (5.8,-0.5) {$C_3$};

\node (yN) at (5,-2.5) {$y_N$};
\draw[red, thick] (5,-2.5) circle (0.5);
\node[red] at (5.8,-2.5) {$C_N$};

\node at (0,-1.5) {$\vdots$};
\node at (5,-1.5) {$\vdots$};

\draw[->] (x1) .. controls (2.5,2.2) and (3.5,1.8) .. (y2);
\draw[->] (x2) .. controls (2.5,1.2) and (3.5,0.8) .. (y3);
\draw[->] (xN) .. controls (2.5,-2) and (3.5,1.5) .. (y1);

\node at (2.5,-3.5) {$y_{N+1}:=y_1,\; C_{N+1}:=C_1$};

\end{tikzpicture} 
\caption{Cyclic permutation of images: $y_{N+1}:=y_1,\; C_{N+1}:=C_1$.} 
\end{figure}
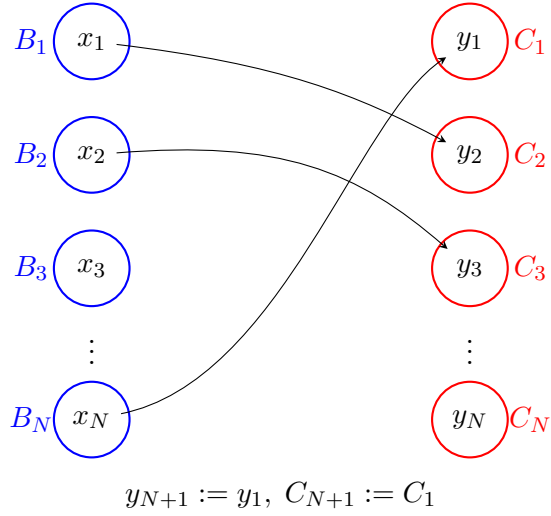

Now, since $T^{*}$ is a minimizer,
\[
\mathcal{M}(T^{*})\leq \mathcal{M}(T).
\]
Then from \eqref{eqn: Monge}, we deduce
\begin{align*}
    \sum_{k=1}^{N}\int_{B_{k}}|x-T^{*}(x)|d\mu(x)\leq\sum_{k=1}^{N}\int_{B_{k}}|x-T(x)|d\mu(x).
\end{align*}
In particular, from the above construction, we have $T^{*}(x_{k})=y_{k}$ and $T^{*}(B_{k})=C_{k}$ along with the cyclically permutation of the image balls under $T$, \eqref{new T}. The application of cyclical monotonicity,  Definition \ref{definition of cyclical sets}, applied to $c(x,y)=|x-y|$ along with $\mathcal{M}(T^{*})\leq\mathcal{M}(T)$, induces the following inequality
\[
\sum_{k=1}^{N}|x_{k}-y_{k}|\leq\sum_{k=1}^{N}|x_{k}-y_{k+\ell}|.
\]

The geometric interpretation from the above cyclical monotonicity inequality is as follows. Consider any closed curve $\mathcal{C}$ in $\text{spt}(\mu)$, parametrized with constant speed, $\{\mathbf{r}(t)\;:\; 0\leq t \leq 1\}$, $\mathbf{r}(0)=\mathbf{r}(1)$. Let $x_{k}=\mathbf{r}\left(\frac{k}{N}\right)$ be $N$ equally spaced points along $\mathcal{C}$ just like in Evans' \cite{evans1997monge}. Now fix $\varepsilon>0$ and an a positive integer $\ell$, and take $N\to \infty$, $\frac{\ell}{N}\to \varepsilon$ in the above cyclical inequality to deduce,
\begin{align*}
    \int_{0}^{1}\left|\mathbf{r}(t)-T^{*}(\mathbf{r}(t))\right|dt&\leq\int_{0}^{1}\left|\mathbf{r}(t)-T^{*}(\mathbf{r}(t+\varepsilon))\right|dt\\
    &=\int_{0}^{1}\left|\mathbf{r}(t-\varepsilon)-T^{*}(\mathbf{r}(t))\right|dt\\
    &=:i(\varepsilon).
\end{align*}
So since $T^{*}$ is a minimizer of \eqref{eqn: Monge}, then $i(\varepsilon)$ has a minimum at $\varepsilon=0$. Therefore,
\[
0=\frac{d}{d\varepsilon}\Big|_{\varepsilon=0}i(\varepsilon)=-\int_{0}^{1}\frac{\left(\mathbf{r}(t)-T^{*}(\mathbf{r}(t))\right)\mathbf{r}^{\prime}(t)}{\left|\mathbf{r}(t)-T^{*}(\mathbf{r}(t))\right|}dt.
\]
Here is where Evans \cite{evans1997monge} made the crucial connection to calculus---gradient of a vector field. Indeed, the above is equivalent to 
\begin{align}\label{line integral}
\oint_{\mathcal{C}}\xi\cdot d\mathbf{r}=0,
\end{align}
for the vector field 
\begin{align}\label{vector field}
    \xi(x):=\frac{T^{*}(x)-x}{\left|T^{*}(x)-x\right|}.
\end{align}
In accordance with \eqref{line integral} holding for all closed curves $\mathcal{C}$ in domain $\mathcal{D}\subset\mathbb{R}^{d}$, the standard zero-circulation criterion implies that $\xi$ is a conservative vector field, and hence there exists a scalar function $f^{*}:\mathbb{R}^{d}\to\mathbb{R}$ such that 
\[
\xi(x)=-\nabla f^{*}(x).
\]
As a consequence this solves an \emph{eikonal equation},
\[
\left|\nabla f^{*}(x)\right|=1
\]
on $\text{spt}(\mu)$. In other words, the direction prescribed by $T^{*}$ at each point $x$ is given by the gradient of a potential scalar function:
\begin{equation}
T^{*}(x) = x - \mathsf{d}^{*}(x) \nabla f^{*}(x).
\label{eq:optimal_direction}
\end{equation}
It is crucial to observe, however, that this derivation yields absolutely no information regarding the distance $\mathsf{d}^{*}(x) := |T^{*}(x) - x|$ along which the transport occurs. 
\begin{remark}
As Evans remarks, the insightful computations done in  \cite{evans1997monge} were first obtained by Monge, who arrived at them through purely geometric reasoning involving developable surfaces. For further details, we refer to the original works of Monge \cite{monge1781memoire} and Dupin \cite{dupin1822applications}. The investigations undertaken by Monge and his disciples into this problem ultimately led to many of their seminal contributions to differential geometry; see, for instance, the historical account by Struik \cite{Struik1988}.
\end{remark}

\subsection{The geometry of paths in optimal tranport}
In the same vein as the closing discussion of the previous section, we now apply the preceding calculations to the geometry of path-dependent transport in optimal transport theory. This line of inquiry formed the core of the author’s doctoral research in 2022, during which he examined an optimal transport problem involving path dependence and congestion effects. The problem itself was originally proposed to him by his advisor, Guillen, in 2018 as a dissertation topic. Then, Cabrera began incorporating path-dependent structures and congestion into optimal transport frameworks. It is worth noting, however, that such ideas, of incorporating a path-dependent transport problem to optimal transport, was not entirely novel; indeed, Villani (\cite{villani2008optimal}) and others had previously explored optimal transport in conjunction with Lagrangian dynamics. In particular, Villani advanced this perspective by formulating a time-dependent variant of the optimal transport problem, thereby enriching the classical theory with dynamic considerations.

Brenier established the structure and uniqueness of optimal transport maps for the quadratic cost, while Benamou and Brenier in \cite{benamou2000computational}  introduced the dynamics fluid-mechanical formulation, showing that optimal transport corresponds to the least-action path between probability measures. Namely, reinterpreted optimal transport dynamically as minimizing kinetic energy along a flow of probability densities. In principle, their theory a priori already included linear paths.

In \cite{villani2003topics} - \cite{villani2008optimal}, Villani explains a construction in optimal transport where an action, $\mathcal{A}$, which measures the cost of displacing along a continuous path $\gamma$, defined on a time interval, is used to consider a cost, $c$, by minimizing the action among paths that go from the initial point of the path, $x$, to the final point of the path, $y$,
\[
c(x, y) = \inf\left\{ \mathcal{A}(\gamma) : \gamma(0) = x, \; \gamma(1) = y, \; \gamma \in C\left([0,1]; \mathbb{R}^d\right) \right\}.
\]
A typical and classical example of such an action is the kinetic energy, 
\[
\mathcal{A}(\gamma) = \int_0^1 \frac{|\dot{\gamma}(t)|^2}{2} \, dt.
\]
More generally, an action is given by Lagrangians as Villani details in \cite{villani2008optimal}. Following this line of thinking heavily, we begin our study of our cost action functional. We cover some results pertaining to a Lagrangian $L(\dot{\gamma}, \gamma, t)$ and its action functional, which is given by the time integral of $L(\dot{\gamma}, \gamma, t)$ along the path, which defines a cost function 
\[
c(\gamma) := \int_0^1 L(\dot{\gamma}, \gamma, t) \, dt,
\]
just like in Villani's general example of $\mathcal{A}(\gamma)$. The action functional will be interpreted as the cost of that path. 

What made Cabrera's frame work on optimal transport and path dependence sufficiently novel was to introduce interaction and congestion effects to the problem.  We will see that in the next section the similarities and differences from what Villani and others did. 
\begin{remark}
    Notice that the action functional $\mathcal{A}$ coincides with the energy functional \eqref{KE}, namely, the Lagrangian, $L(\dot{\gamma},\gamma, t)$ equals the energy in \eqref{kinetic energy}. Furthermore, in \cite{cabrera2021optimal}, the author similarly applied a version of an ``effective cost" function to solve the MKP on paths. 
\end{remark}
\subsection{Optimal transport, paths, and congestion} The optimal transport problem that incorporates paths and congestion was dealt with in \cite{cabrera2021optimal}-\cite{Cabrera2022}. Moreover, the problem  studied in \cite{cabrera2021optimal}, namely the minimization of \eqref{eqn: OT paths}, where the cost $c(\gamma)$ is the kinetic energy \eqref{KE}, while the distribution of paths given by $\pi$, $\int\mathcal{U}d\pi$ in \eqref{interaction functional}, uses the Gaussian interaction potential $\mathcal{K}_{G}$ and is thus: 
\[ \begin{aligned} \min_{\pi\in\Pi_{\text{path}}(\mu,\nu)} \quad & \int_{\Omega} c(\gamma)\, d\pi(\gamma) +\tfrac{1}{2}\iint_{\Omega\times\Omega}\mathcal{U}(\gamma,\sigma) d\pi(\gamma)d\pi(\sigma)\\ \text{subject to}\quad & (e_0)_\# \pi = \mu,\\ & (e_1)_\# \pi = \nu. \end{aligned} \] 

Furthermore, looking at the ``optimal transport with paths problem" and omitting the interaction term from above for a moment, $\iint\mathcal{U}d\pi d\pi$ where $\mathcal{U}$ is defined in \eqref{interaction functional}, we have the following. If a path $\gamma$ with end-points $\gamma(0)=x, \gamma(1)=y$ in the path space $\Omega$ appears in an optimal plan, with kinetic cost $c$ defined by
\begin{align}\label{KE}
c(\gamma):=\int_{0}^{1}\tfrac{1}{2}|\dot{\gamma}|^{2}\;dt,
\end{align}
then one would expect that \[ c(\gamma)=\widetilde{c}\bigl(\gamma(0),\gamma(1)\bigr). \] Here $\widetilde{c}$ denotes what we shall call the \emph{end-point cost}, \[ \widetilde{c}(x,y) := \inf\Bigl\{\, c(\gamma)\;:\;\gamma(0)=x,\ \gamma(1)=y \Bigr\}. \] Such paths will be said to be \emph{$c$-minimal}. This is obtained by adapting the line of reasoning developed by Villani and others for the action functional $\mathcal{A}$ appearing in the previous section. In view of this dynamics, we have a solution to the classical Kantorovich problem that incorporates the geometry of paths.

Reintroducing the interaction term due to congestion, $\frac12 \iint_{\Omega\times\Omega} \mathcal U(\gamma,\sigma)\, d\pi(\gamma)d\pi(\sigma)$, we now identify the effective cost associated with a minimizer of the interacting path-space problem. Suppose that $\pi_0\in\Pi_{\mathrm{path}}(\mu,\nu)$ is a minimizer of \eqref{eqn: OT paths}. For any $\pi\in\Pi_{\mathrm{path}}(\mu,\nu),$ consider the admissible variation \[ \pi_\varepsilon := (1-\varepsilon)\pi_0+\varepsilon\pi, \qquad 0\leq\varepsilon\leq1. \] Since $\Pi_{\mathrm{path}}(\mu,\nu)$ is convex \cite{Cabrera2022}, $\pi_\varepsilon$ has the same prescribed endpoint marginals and is therefore an admissible competitor. By the minimality of $\pi_0$, we have that
\begin{align} \label{first variation effective cost} 0 \leq \left. \frac{d}{d\varepsilon} \mathcal J(\pi_\varepsilon) \right|_{\varepsilon=0^+}. \end{align} 
Set $\delta\pi:=\pi-\pi_0.$ The first variation of the single-path term is \[ \left. \frac{d}{d\varepsilon} \right|_{\varepsilon=0} \int_\Omega c(\gamma)\,d\pi_\varepsilon(\gamma) = \int_\Omega c(\gamma)\,d\delta\pi(\gamma). \] For the interaction term, symmetry of $\mathcal U$ gives \begin{align*} & \left. \frac{d}{d\varepsilon} \right|_{\varepsilon=0} \frac12 \iint_{\Omega\times\Omega} \mathcal U(\gamma,\sigma)\, d\pi_\varepsilon(\gamma)d\pi_\varepsilon(\sigma) \\ &\qquad= \frac12 \iint_{\Omega\times\Omega} \mathcal U(\gamma,\sigma)\, d\delta\pi(\gamma)d\pi_0(\sigma) + \frac12 \iint_{\Omega\times\Omega} \mathcal U(\gamma,\sigma)\, d\pi_0(\gamma)d\delta\pi(\sigma) \\ &\qquad= \iint_{\Omega\times\Omega} \mathcal U(\gamma,\sigma)\, d\delta\pi(\gamma)d\pi_0(\sigma). \end{align*} 

Thus the factor $\frac12$ appearing in the quadratic interaction energy disappears in the first variation. For how much a path costs, $c(\gamma)$ given as in\eqref{KE}, this motivates the following effective cost: \begin{align} \label{effective cost} c_{\pi_0}(\gamma) := c(\gamma) + \int_\Omega \mathcal U(\gamma,\sigma)\,d\pi_0(\sigma). \end{align} Combining the two first variations in \eqref{first variation effective cost} yields \[ 0 \leq \int_\Omega c_{\pi_0}(\gamma)\, d(\pi-\pi_0)(\gamma), \] and therefore
\begin{align} \label{linearized path optimality} \int_\Omega c_{\pi_0}(\gamma)\,d\pi_0(\gamma) \leq \int_\Omega c_{\pi_0}(\gamma)\,d\pi(\gamma), \qquad \text{for every } \pi\in\Pi_{\mathrm{path}}(\mu,\nu). \end{align} Hence the nonlinear optimality of $\pi_0$ for the original interacting functional implies that $\pi_0$ minimizes the linearized path-space functional \[ \pi \longmapsto \int_\Omega c_{\pi_0}(\gamma)\,d\pi(\gamma) \] over $\Pi_{\mathrm{path}}(\mu,\nu)$. We now define the corresponding effective endpoint cost by \begin{align} \label{end-point cost} \widetilde c_{\pi_0}(x,y) := \inf\left\{ c_{\pi_0}(\gamma): \gamma(0)=x,\; \gamma(1)=y \right\}. \end{align} 
The effective cost \eqref{effective cost} incorporates both contributions appearing in the interacting path-space problem \eqref{eqn: OT paths}: the single-path action $c$, which contains the prescribed background potential $V$ through \eqref{kinetic energy}, and the interaction generated by the distribution of paths through $\mathcal U$. More importantly, \eqref{linearized path optimality} shows that $c_{\pi_0}$ is not merely a convenient rewriting of the energy; it is precisely the cost generated by the first variation of the nonlinear path-space functional at $\pi_0$. 

This effective path cost provides the variational link between the two parts of the present paper. Minimization with prescribed endpoints leads to the effective endpoint cost $\widetilde c_{\pi_0}$ in \eqref{end-point cost}, which will be used in the endpoint cyclical-monotonicity argument and in the construction of the conservative initial-momentum field. On the other hand, endpoint-preserving variations of the original interacting path-space functional yield the Euler--Lagrange dynamics and the corresponding kinetic equation. Thus the endpoint geometry and the phase-space dynamics are two consequences of the same underlying path-space variational problem.

The next result quantifies the motivations behind defining the effective and end-point cost\eqref{effective cost}-\eqref{end-point cost}; and, at the same time illustrates that a Kantorovich problem with congestion has a solution in the path space  and projects onto a Euclidean solution. The proof was done in \cite{cabrera2021optimal}, but we prove it here to demonstrate the reduction from path-space optimality to endpoint optimality.

\begin{theorem} \label{thm:endpoint reduction} Let $c(\gamma)$ be as in \eqref{KE}, and let $\pi_0\in\Pi_{\mathrm{path}}(\mu,\nu)$ be a minimizer of the interacting path-space functional
 \[ \mathcal J(\pi) = \int_\Omega c(\gamma)\,d\pi(\gamma) + \frac12 \iint_{\Omega\times\Omega} \mathcal U(\gamma,\sigma)\, d\pi(\gamma)d\pi(\sigma), \]
 where $\mathcal U$ is symmetric. Define the linearized effective path cost \eqref{effective cost} and the associated effective endpoint cost \eqref{end-point cost}. Assume that the endpoint minimization problem defining $\widetilde c_{\pi_0}$ admits measurable minimizing selections, or, more generally, that measurable $\varepsilon$-minimizing selections exist for every $\varepsilon>0$. Then the following statements hold. \begin{enumerate} \item The optimal path plan $\pi_0$ minimizes the linearized functional \[ \pi \longmapsto \int_\Omega c_{\pi_0}(\gamma)\,d\pi(\gamma) \] over $\Pi_{\mathrm{path}}(\mu,\nu)$. \item For $\pi_0$-almost every $\gamma\in\Omega$, \begin{equation} \label{eq:pathwise endpoint equality} c_{\pi_0}(\gamma) = \widetilde c_{\pi_0}\bigl(\gamma(0),\gamma(1)\bigr). \end{equation} \item The endpoint coupling \begin{equation} \label{eq:endpoint coupling eta star} \eta_{0} := (e_0,e_1)_\#\pi_0 \end{equation} minimizes \[ \eta \longmapsto \int_{\mathbb R^d\times\mathbb R^d} \widetilde c_{\pi_0}(x,y)\,d\eta(x,y) \] over $\Pi(\mu,\nu)$. In particular, if \[ \pi_0=(x\mapsto\gamma_x)_\#\mu, \qquad \gamma_x(0)=x, \qquad \gamma_x(1)=T^*(x), \] where $T^{*}$ satisfies \eqref{optimal solution},
 then 
 \[ \eta_{0} = (\operatorname{id},T^*)_\#\mu \] 
 and 
 \begin{equation} \label{eq:deterministic endpoint equality} c_{\pi_0}(\gamma_x) = \widetilde c_{\pi_0}\bigl(x,T^*(x)\bigr), \qquad \text{for $\mu$-a.e. }x. 
 \end{equation} 
 \end{enumerate} 
 \end{theorem} 
 \begin{proof} 
 Firstly, Linearize the optimal $\pi_0$. To this end, let $\pi\in\Pi_{\mathrm{path}}(\mu,\nu)$ be arbitrary and, for $\varepsilon\in[0,1]$, define the variation
 \[ \pi_\varepsilon = (1-\varepsilon)\pi_0+\varepsilon\pi. \] 
 Since $\Pi_{\mathrm{path}}(\mu,\nu)$ is convex,$\pi_\varepsilon\in\Pi_{\mathrm{path}}(\mu,\nu)$. By the minimality of $\pi_0$, we thus have
 \[ \mathcal J(\pi_\varepsilon) \geq \mathcal J(\pi_0). \] 
 Consequently, \[ 0 \leq \left. \frac{d}{d\varepsilon} \mathcal J(\pi_\varepsilon) \right|_{\varepsilon=0^+}. \] 
 Writing $\delta\pi:=\pi-\pi_0,$ the first variation of the single-path term is calculated as follows
 \[ \left. \frac{d}{d\varepsilon} \int_\Omega c(\gamma)\,d\pi_\varepsilon(\gamma) \right|_{\varepsilon=0} = \int_\Omega c(\gamma)\,d\delta\pi(\gamma). \] 
 For the interaction term, symmetry of $\mathcal U$ gives
 \begin{align*}  \left. \frac{d}{d\varepsilon} \frac12 \iint_{\Omega\times\Omega} \mathcal U(\gamma,\sigma)\, d\pi_\varepsilon(\gamma)d\pi_\varepsilon(\sigma) \right|_{\varepsilon=0}\;= \iint_{\Omega\times\Omega} \mathcal U(\gamma,\sigma)\, d\delta\pi(\gamma)d\pi_0(\sigma). \end{align*} 
 Hence \begin{align*} 0 &\leq \int_\Omega \left[ c(\gamma) + \int_\Omega \mathcal U(\gamma,\sigma)\,d\pi_0(\sigma) \right] d(\pi-\pi_0)(\gamma) \\ &= \int_\Omega c_{\pi_0}(\gamma)\, d(\pi-\pi_0)(\gamma). \end{align*} Therefore \begin{equation} \label{eq:linearized optimality} \int_\Omega c_{\pi_0}(\gamma)\,d\pi_0(\gamma) \leq \int_\Omega c_{\pi_0}(\gamma)\,d\pi(\gamma) \end{equation} for every $\pi\in\Pi_{\mathrm{path}}(\mu,\nu)$. Thus $\pi_0$ minimizes the linearized effective cost.

 Secondly, we shall show that optimal trajectories minimize the effective cost between their own endpoints. Indeed, by definition of $\widetilde c_{\pi_0}$, 
 \begin{equation} \label{eq:endpoint lower bound} \widetilde c_{\pi_0}\bigl(\gamma(0),\gamma(1)\bigr) \leq c_{\pi_0}(\gamma) \end{equation}
 for every admissible path $\gamma$. We claim that equality holds for $\pi_0$-almost every $\gamma$. Suppose, towards sake of contradiction, that the set 
 \[ A = \left\{ \gamma\in\Omega : c_{\pi_0}(\gamma) > \widetilde c_{\pi_0}\bigl(\gamma(0),\gamma(1)\bigr) \right\} \] 
 has positive $\pi_0$-measure. Then, for some $\delta>0$, the set \[ A_\delta = \left\{ \gamma\in\Omega : c_{\pi_0}(\gamma) \geq \widetilde c_{\pi_0}\bigl(\gamma(0),\gamma(1)\bigr) +\delta \right\} \] has positive $\pi_0$-measure. For each endpoint pair $(x,y)$ arising from $A_\delta$, choose, by coercivity and unifomr convergence \cite{cabrera2021optimal}, an admissible path $\widehat\gamma_{x,y}$ joining $x$ to $y$ such that \[ c_{\pi_0}(\widehat\gamma_{x,y}) \leq \widetilde c_{\pi_0}(x,y) +\frac{\delta}{2}. \] Define the measurable map $S:\Omega\to\Omega$ that changes $\gamma$ to $\widehat{\gamma}$ while having their endpoints fixed and equivalent; otherwise the path remains unchanged, by
\[
S(\gamma)
:=
\begin{cases}
\widehat\gamma_{\gamma(0),\gamma(1)},
& \gamma\in A_\delta,\\[0.3em]
\gamma,
& \gamma\notin A_\delta,
\end{cases}
\]
and set the measure preserving map on paths as
\[
\widehat\pi:=S_\#\pi_0.
\]

By construction, the replacement path has the same endpoints as the
original path. Hence
\[
e_i\circ S=e_i,
\qquad i=0,1.
\]
It follows that
\[
(e_i)_\#\widehat\pi
=
(e_i)_\#(S_\#\pi_0)
=
(e_i\circ S)_\#\pi_0
=
(e_i)_\#\pi_0,
\qquad i=0,1.
\]
Therefore $(e_0)_\#\widehat\pi=\mu,
\;
(e_1)_\#\widehat\pi=\nu,$
and therefore $\widehat\pi\in\Pi_{\mathrm{path}}(\mu,\nu).$

Moreover, for every $\gamma\in A_\delta$, the mapping under $S$ provides the estimate, $c_{\pi_{0}}(S(\gamma))<\widetilde{c}_{\pi_{0}}(\gamma(0),\gamma(1))+\delta/2$. Hence, subtracting $\delta/2$ on both sides of the inequality in $A_{\delta}$ yields,
\[
c_{\pi_{0}}(\gamma)-\delta/2\geq \widetilde{c}_{\pi_{0}}(\gamma(0),\gamma(1))+\delta/2\geq c_{\pi_{0}}(S(\gamma)).
\]
Since $S(\gamma)=\gamma$ outside $A_\delta$, and applying the above estimates we obtain
\begin{align*}
\int_\Omega c_{\pi_0}(\gamma)\,d\widehat\pi(\gamma)
&=
\int_\Omega c_{\pi_0}(S(\gamma))\,d\pi_0(\gamma)\\
&\leq
\int_\Omega c_{\pi_0}(\gamma)\,d\pi_0(\gamma)
-\frac{\delta}{2}\pi_0(A_\delta)\\
&<
\int_\Omega c_{\pi_0}(\gamma)\,d\pi_0(\gamma),
\end{align*}
as $\pi_0(A_\delta)>0$. This is the desired contradiction to
\eqref{eq:linearized optimality}.
 Hence \[ c_{\pi_0}(\gamma) = \widetilde c_{\pi_0}\bigl(\gamma(0),\gamma(1)\bigr) \] for $\pi_0$-almost every $\gamma$.

 Finally, we apply an optimality of the endpoint coupling. Define \[ \eta_{0}: = (e_0,e_1)_\#\pi_0. \] Since $ (e_0)_\#\pi_0=\mu, \; (e_1)_\#\pi_0=\nu, $ we have $\eta_{0}\in\Pi(\mu,\nu). $ 
 According to the optimal trajectories that minimize the effective cost \eqref{end-point cost} between their end-points, and the definition of push-forward, \begin{align} \int_\Omega c_{\pi_0}(\gamma)\,d\pi_0(\gamma) &= \int_\Omega \widetilde c_{\pi_0} \bigl(\gamma(0),\gamma(1)\bigr) \,d\pi_0(\gamma) \nonumber\\ &= \int_{\mathbb R^d\times\mathbb R^d} \widetilde c_{\pi_0}(x,y)\, d\eta_{0}(x,y). \label{eq:path endpoint equality integrated} \end{align} 
 Now let $\eta\in\Pi(\mu,\nu)$ be arbitrary. For each $\varepsilon>0$, choose  a path $\gamma^\varepsilon_{x,y}$ joining $x$ to $y$ such that
 \begin{align}\label{inq:linearized inq}
 c_{\pi_0}(\gamma^\varepsilon_{x,y}) \leq \widetilde c_{\pi_0}(x,y)+\varepsilon. 
 \end{align}
 Define the corresponding path plan by \[ \pi^\varepsilon = \bigl((x,y)\mapsto\gamma^\varepsilon_{x,y}\bigr)_\#\eta. \]
 Because every selected path joins $x$ to $y$, $(e_0)_\#\pi^\varepsilon=\mu,$ and $(e_1)_\#\pi^\varepsilon=\nu,$ so $\pi^\varepsilon \in \Pi_{\mathrm{path}}(\mu,\nu).$ By the linearized optimality of $\pi_0$, we have
 \[ \int_\Omega c_{\pi_0}\,d\pi_0 \leq \int_\Omega c_{\pi_0}\,d\pi^\varepsilon. \] 
 Using the defining property of the selected paths $\gamma_{x,y}^{\varepsilon}$ which satisfy \eqref{inq:linearized inq}, we aqcuire the following estimate
 \[ \int_\Omega c_{\pi_0}\,d\pi^\varepsilon \leq \int_{\mathbb R^d\times\mathbb R^d} \widetilde c_{\pi_0}(x,y)\,d\eta(x,y) +\varepsilon. \] 
 Then combining this with \eqref{eq:path endpoint equality integrated} gives \[ \int \widetilde c_{\pi_0}\,d\eta_{0} \leq \int \widetilde c_{\pi_0}\,d\eta +\varepsilon. \] Letting $\varepsilon\downarrow0$ decrease to zero yields \[ \int_{\mathbb R^d\times\mathbb R^d} \widetilde c_{\pi_0}(x,y)\,d\eta_{0}(x,y) \leq \int_{\mathbb R^d\times\mathbb R^d} \widetilde c_{\pi_0}(x,y)\,d\eta(x,y). \] Since $\eta\in\Pi(\mu,\nu)$ was arbitrary, $\eta_{0}$ is optimal for the effective endpoint cost $\widetilde c_{\pi_0}$ in \eqref{end-point cost}. Finally, if the following is met
 \begin{align}\label{deterministic plan}
    \pi_0=(x\mapsto\gamma_x)_{\#}\mu,
\end{align}
where
\begin{align}\label{conditions of paths}
\gamma_x(0)=x,
\qquad
\gamma_x(1)=T^*(x),
\qquad
T^*_{\#}\mu=\nu,
\end{align}
 then 
\begin{align}\label{endpoint plan}
    \eta_0
    &:=(e_0,e_1)_{\#}\pi_0
      =(\operatorname{id},T^*)_{\#}\mu.
\end{align}
Applying \eqref{eq:pathwise endpoint equality} to $\gamma=\gamma_x$ gives \[ c_{\pi_0}(\gamma_x) = \widetilde c_{\pi_0}\bigl(x,T^*(x)\bigr) \] for $\mu$-almost every $x$. 
\end{proof}

In view of this result, we characterize such a minimizer through the dual problem. 
\begin{lemma}\label{dual problem} 
Let $\pi_0$ be a minimizer of the optimal transport with interaction and congestion problem \eqref{eqn: OT paths}. Then there exist functions \[ \varphi,\psi:\mathbb{R}^d\to\mathbb{R} \] such that \[ \varphi(\gamma(0))+\psi(\gamma(1)) \leq c_{\pi_0}(\gamma), \qquad \forall\,\gamma\in\Omega, \] and \[ \varphi(\gamma(0))+\psi(\gamma(1)) = c_{\pi_0}(\gamma), \qquad \text{for }\pi_0\text{-a.e. }\gamma\in\Omega. \] \end{lemma}
\begin{proof}
    Although one can find the proof in \cite{cabrera2021optimal} or in a recently submitted article (that introduces a variant of OT) \cite{cabrera2026minimax}, we shall provide at least a sketch of the proof to shed light on the ideas stemming from convex optimization and linear programming. 

    Apply Lagrangians multipliers $\varphi, \psi, \lambda$ in the Lagrangian functional, $\Lambda$.  
Suppose $\pi_{0}$ is a miminizer of \eqref{eqn: OT paths}. Using Lagrangian multipliers, $\varphi, \psi\in C_{c}^{0}(\mathbb{R}^{d})\times C_{c}^{0}(\mathbb{R}^{d})$, consider the Lagrangian function
\begin{align*}
\Lambda(\pi, \varphi, \psi, \lambda)&:=\int_{\Omega}c_{\pi_{0}}(\gamma)d\pi(\gamma)+\int_{\mathbb{R}^{d}}\varphi(x)d\mu(x)-\int_{\mathbb{R}^{d}}\varphi(\gamma(0))d\pi(\gamma)\\
&+\int_{\Omega}\psi(y)d\nu(y)-\int_{\Omega}\psi(\gamma(1))d\pi(\gamma)+\int_{\Omega}\lambda(\gamma)d\pi(\gamma).
\end{align*}
Rearranging, we get
\begin{align*}
    \Lambda(\pi, \varphi, \psi, \lambda)&:=\int_{\Omega}c(\gamma)d\pi(\gamma)+\frac{1}{2}\int_{\Omega}\int_{\Omega}\mathcal{U}(\gamma, \sigma)d\pi(\gamma)d\pi(\sigma)\\
    &+\int_{\mathbb{R}^{d}}\varphi(x)d\mu(x)+\int_{\mathbb{R}^{d}}\psi(y)d\nu(y)+\int_{\Omega}\lambda(\gamma)-\left(\varphi(\gamma(0))+\psi(\gamma(1))\right)d\pi(\gamma)\\\
  &=\int_{\Omega}c(\gamma)d\pi(\gamma)+\frac{1}{2}\int_{\Omega}\int_{\Omega}\mathcal{U}(\gamma,\sigma)d\pi(\gamma)d\pi(\sigma)\\
  &+\int\varphi(x)d\mu(x)+\int\psi(y)d\nu(y)\int_{\Omega}\lambda(\gamma)-\left(\varphi(\gamma(0))+\psi(\gamma(1))\right)d\pi(\gamma)
\end{align*}
where we included the marginal constraints $(e_{0})_{\sharp}\pi=\mu$ and $(e_{1})_{\sharp}\pi=\nu$, and the nonnegativity constraint due to the measure, $\int \lambda(\gamma) d\pi(\gamma)$.  Then we compute the gradient with respect to $\pi$ in $\gamma$. Note that we get linear conditions on $\pi$. So according to \cite[Corollary 4.5]{cabrera2021optimal}, we take a ``smooth" curve $\pi(s)$ and take derivatives to get an expression regarding  tangent vectors $\dot{\pi}$ and an expression in the functional \eqref{eqn: OT paths} term regarding an integral against $\pi_{0}$, namely, 
\begin{align*}
    \frac{d}{ds}\Big|_{s=0}\Lambda(\pi(s), \varphi, \psi, \lambda)&=\int_{\Omega}c(\gamma)d\dot{\pi}(\gamma)+\;\int_{\Omega}\int_{\Omega}\mathcal{U}(\gamma, \sigma)d\pi_{0}(\sigma)d\dot{\pi}(\gamma)\\
    &+\int_{\Omega}\lambda(\gamma)-\left(\varphi(\gamma(0))+\psi(\gamma(1))\right)d\dot{\pi}(\gamma).
\end{align*}
Therefore, since we have a \emph{critical point}, this indicates that we get a zero functional. Then minimality---the KKT conditions \cite{rockafellar2015convex}, actually---there exist a pair $\varphi, \psi$ and $\lambda\geq 0 $ such that 
\begin{align*}
c(\gamma)+\int_{\Omega}\mathcal{U}(\gamma, \sigma)d\pi_{0}(\sigma)+\lambda(\gamma)-\left(\varphi(\gamma(0))+\psi(\gamma(1))\right)&=0.
\end{align*}
Moreover, $\lambda\equiv 0$ in the support of $\pi$.

Notice that from the fact that $\lambda$ is positive in the above equality, one can see then that the sum of the $\varphi$ and $\psi$ is less than the integral term, $c(\gamma)+\int\mathcal{U}d\pi$, everywhere, and exactly equal wherever $\lambda$ vanishes in the support of $\pi$.
\end{proof}
In the next section, we formulate a similar construction of that of Evans, and this shows that equality holds, \eqref{eq:deterministic endpoint equality}, on the endpoints selected by the optimal path measure and consequently identifies the endpoint coupling induced by $\pi_0$ as an optimal coupling for the effective cost $\widetilde c_{\pi_0}$.

\subsection{Optimal Transport with paths and eikonal sub-solutions}\label{OT and eikonal subsolutions} We are in a position to apply the heuristics calculations of Evans \cite{evans1997monge} to our problem in a general setting. We execute these calculations, in the general sense firstly; before actually showing that $W_{\pi_{0}}$ in \eqref{self consistent Gaussian potential} satisfies a viscosity subsolution of an eikonal type, to optimal transport and gain a conservative vector field (see Corollary \ref{viscosity solution}). Soner and Yan in \cite{soner2024viscosity} applied the Wasserstein metric \eqref{eq:OT} with the quadratic cost to get an eikonal type equation from Hamilton-Jacobi dynamic equations. In particular, they obtained a viscosity solution of this eikonal type equation. In the numerical analysis aspect of optimal transport, Hamfeldt-Turnquist in \cite{hamfeldt2021convergence} analyzed an optimal transport problem on the sphere. Despite the absence of a comparison principle and uniqueness, they obtain a convergent PDE-based numerical method, and approximate a Monge-Ampere type equation as well as an eikonal one.

This idea of associating optimal transport with partial differential equations is not new at all; in fact, it is quite natural. The topic of applying optimal transport with PDEs---specifically the Monge-Ampere type PDE---to study the issues arising from existence and regularity is due to several researchers, namely; to name a fair few are the following, Caffarelli \cite{Caffarelli1992}-\cite{caffarelli1999monge}, Figalli-Kim-McCann \cite{Figalli2013} , Guillen-Kitagawa \cite{guillen2015local}, L. Ambrosio, Y. Bremier, L.C. Evans, M. Feldman \cite{FeldmanMcCann2002}, W. Gangbo, R. McCann, Ma-N. Trudinger- X.J. Wang \cite{MaTrudingerWang2005}, X.J. Wang \cite{TrudingerWang2001}, and J. Urbas \cite{Urbas1997}. Of course, many other researchers have contributed to this field, but the topic of optimal transport and PDEs is rather vast and somewhat fragmented; so we selected to include a short list.

The following calculations illustrate a further application of connecting optimal transport techniques with partial differential equations. We apply optimal transport with paths and congestion to acquire a viscosity sub-solution. Consider the interaction, congestion term via the Gaussian interaction
\begin{align}\label{congestion term}
    \int_{\Omega}\mathcal{U}(\gamma, \sigma)d\pi(\sigma):=\int_{\Omega}\int_{0}^{1}\lambda e^{-\beta|\gamma(t)-\sigma(t)|^{2}}\;dtd\pi(\sigma),\; \lambda,\beta>0.
\end{align}
Indeed, as $\mathcal{U}(\gamma,\sigma)$ is given by \eqref{interaction functional} and $\mathcal{K}_{G}$ by the Gaussian, we arrive at the above \eqref{congestion term}.

To prove that the latter property, or rather that \eqref{self consistent Gaussian potential}, yields a viscosity subsolution of the eikonal equation, we develop an argument analogous to that of Evans' approach based on a cyclical monotonicity argument; see \cite[Sec.~8.3]{evans1997monge}. We first recall the deterministic structure of the optimal path plan established in \cite{cabrera2021optimal}.

Let $\pi_0$ be an optimal path plan for \eqref{eqn: OT paths}. By
\cite{cabrera2021optimal}, suppose that $\pi_0$ is induced by a Borel
path map, namely \eqref{deterministic plan}-\eqref{conditions of paths}.
As before, let the evaluation map be defined as $e_t(\gamma):=\gamma(t),
\; 0\leq t\leq1,$
and define
\begin{align}\label{time marginals optimal plan}
    \mu_t:=(e_t)_{\#}\pi_0,
    \qquad 0\leq t\leq1.
\end{align}
In particular, $\mu_0=\mu,
\;
\mu_1=\nu.$
The endpoint plan associated with $\pi_0$ is given by \eqref{endpoint plan}.

By the effective cost function which induces its corresponding end-point cost function in
\eqref{effective cost}--\eqref{end-point cost}, $\eta_0$ is optimal
for the endpoint cost
\[
\widetilde c_{\pi_0}(x,y)
:=
\inf_{\substack{\gamma(0)=x\\\gamma(1)=y}}
c_{\pi_0}(\gamma).
\]
Moreover, since $\gamma_{x}$ is the chosen minimal path satisfying the above constraints, we get  \eqref{eq:deterministic endpoint equality}. Consequently,
$\operatorname{spt}(\eta_0)$ is
$\widetilde c_{\pi_0}$-cyclically monotone.

Next, let $r$ be a constant speed parameterization
\[
r:[0,1]\longrightarrow X
\]
which is a $C^1$ closed curve such that its initial and final point meet, $r(0)=r(1)$,
and suppose that its coordinate points are concentrated on the support of $\eta_{0}$:
\[
(r(s),T^*(r(s)))\in\operatorname{spt}(\eta_0)
\qquad\text{for every }s\in[0,1].
\]
Thus the relevant closed curve is the curve $r$ in the source
space; it is not an optimal trajectory $\gamma_x$.

For $N\geq1$, mimicking Evan's construction \cite{evans1997monge}, set the source and target points as follows:
\[
x_k:=r\left(\frac{k}{N}\right),
\qquad
y_k:=T^*(x_k),
\qquad
k=0,\ldots,N-1.
\]
Fix an integer $\ell$, with indices understood modulo $N$. Since the support of the projector measure $\eta_{0}$,
$\operatorname{spt}(\eta_0)$ is
$\widetilde c_{\pi_0}$-cyclically monotone, the cyclic permutation
$k\mapsto k+\ell$ gives
\begin{align}\label{discrete-cyclic-inequality}
\sum_{k=0}^{N-1}
\widetilde c_{\pi_0}(x_k,y_k)
\leq
\sum_{k=0}^{N-1}
\widetilde c_{\pi_0}(x_k,y_{k+\ell}).
\end{align}
Equivalently,
\begin{align}\label{discrete-cyclic-inequality-2}
\sum_{k=0}^{N-1}
\widetilde c_{\pi_0}
\bigl(x_k,T^*(x_k)\bigr)
\leq
\sum_{k=0}^{N-1}
\widetilde c_{\pi_0}
\bigl(x_k,T^*(x_{k+\ell})\bigr).
\end{align}

Notice that the path appearing implicitly on the right-hand side is
not $\gamma_{x_{k+\ell}}$. Indeed, it is its initial point:
\[
\gamma_{x_{k+\ell}}(0)=x_{k+\ell},
\]
whereas the competing endpoint pair in
\eqref{discrete-cyclic-inequality-2} is $\left(x_k,T^*(x_{k+\ell})\right).$
If a path representative is required, one must instead choose a
minimizer
\[
\gamma_{k,\ell}^{\,N}
\in
\operatorname*{argmin}_{\substack{
\gamma(0)=x_k\\
\gamma(1)=T^*(x_{k+\ell})}}
c_{\pi_0}(\gamma),
\]
so that the following costs coincide:
\[
c_{\pi_0}(\gamma_{k,\ell}^{\,N})
=
\widetilde c_{\pi_0}
\bigl(x_k,T^*(x_{k+\ell})\bigr).
\]

As $\widetilde c_{\pi_0}$ is continuously
differentiable \cite{cabrera2021optimal} in its first variable on a neighborhood of the
relevant endpoint pairs, then the composition $T^*\circ r$ is continuous.
Choose $\ell=\ell_N$ such that
\[
\frac{\ell_N}{N}\longrightarrow\varepsilon.
\]
Dividing \eqref{discrete-cyclic-inequality-2} by $N$ and passing to
the limit in the corresponding Riemann sums yields
\begin{align}\label{Evans limiting inequality}
\begin{split}
\int_0^1
\widetilde c_{\pi_0}
\bigl(r(s),T^*(r(s))\bigr)\;ds
&\leq
\int_0^1
\widetilde c_{\pi_0}
\bigl(r(s),T^*(r(s+\varepsilon))\bigr)\;ds\\
&=\int_0^1 \widetilde{c}_{\pi_{0}}\bigr(r(s-\varepsilon), T^*(r(s))\bigl)\;ds\\
&=:I(\varepsilon),
\end{split}
\end{align}
where the parameter is interpreted periodically.

Notice that we applied the same trick of Evans \cite{evans1997monge} of changing variables $s\mapsto s-\varepsilon$ on the right-hand side of \eqref{Evans limiting inequality}. So since $\pi_{0}$ is the optimal path measure of  \eqref{eqn: OT paths}, then according to Theorem \ref{thm:endpoint reduction}, $\eta_{0}$ solves the Monge-Kantorovich problem \eqref{eq:OT} and consequently \eqref{optimal solution}. Therefore $T^*$ is the optimal mapping appearing in \eqref{eq:deterministic endpoint equality}, and thus we acquire

\[
I(0)\leq I(\varepsilon),
\]

Thus $\varepsilon=0$ is a minimum of $I$. Under the above
regularity assumptions, differentiation under the integral sign
yields
\begin{align*}
0
=
I'(0)
&=
-\int_0^1
D_x\widetilde c_{\pi_0}
\bigl(r(s),T^*(r(s))\bigr)
\cdot r'(s)\,ds.
\end{align*}
Consequently,
\begin{align}\label{zero circulation effective field}
\oint_{\mathcal C}
D_x\widetilde c_{\pi_0}(x,T^*(x))\cdot dx
=0.
\end{align}

Hence, provided this identity holds for every $C^1$ closed curve in
a simply connected open set $U$ of $X$ and the vector field
\[
F_{\pi_0}(x)
:=
D_x\widetilde c_{\pi_0}(x,T^*(x))
\]
is continuous on $U$, the classical conservative-field from calculus
implies that there exists a scalar potential
$u^*:U\to\mathbb R$ such that
\begin{align}\label{effective conservative potential}
\nabla u^*(x)
=
D_x\widetilde c_{\pi_0}(x,T^*(x)).
\end{align}

The preceding Evans-type argument and the following Gaussian
calculation play distinct but complementary roles. Cyclical
monotonicity of the optimal endpoint plan shows that the effective
endpoint field
\[
x\longmapsto
D_x\widetilde c_{\pi_0}(x,T^*(x))
\]
has zero circulation and is therefore conservative under the stated
regularity and topological hypotheses. This line of argument is what will make the derivation of a viscosity subsolution of an eikonal equation possible. See Section \ref{subsec:gaussian-effective-action}.


\subsection{Gaussian Interaction and PDE structure}\label{sec:gaussian-eikonal}


This section builds on Section \ref{OT and eikonal subsolutions} to generate the necessary and sufficient conditions to establish that $W_{\pi}$, as given in \eqref{self consistent Gaussian potential}, for the effective cost function \eqref{effective action Gaussian} with corresponding end-point cost function \eqref{effective endpoint Gaussian} satisfies a viscosity subsolution of eikonal equation. For the optimal path plan $\pi_0$, recall that $\mu_t$ is given by \eqref{time marginals optimal plan} and recall the Gaussian kernel \eqref{kernel}, is given as the Gaussian distribution $\mathcal K_G(z)
=
\lambda e^{-\beta|z|^2},
\;
\lambda,\beta>0.$

The potential \eqref{self consistent Gaussian potential} is the
interaction potential generated by the time marginals of the optimal
path plan. It should be distinguished from the conservative field
obtained from endpoint cyclical monotonicity. In particular, without
an additional identity relating
\[
D_x\widetilde c_{\pi_0}\bigl(x,T^*(x)\bigr)\quad\text{to}\quad \nabla_x W_{\pi_0}(t,x),
\]
these two fields should not be identified. Namely, they act on different fields depeding on regularity.

The eikonal subsolution property established below for
$W_{\pi_0}$ follows instead from the regularity and uniform gradient
estimate for the Gaussian interaction potential
\eqref{self consistent Gaussian potential}.

We first express the effective path cost in terms of
the time marginals of the optimal path measure.
Recall that \eqref{time marginals optimal plan}
and that the Gaussian interaction potential is
given by \eqref{self consistent Gaussian potential}.
By Tonelli's theorem and the definition of
the pushforward measure,
\begin{align*}
\int_\Omega
\mathcal U(\gamma,\sigma)\,d\pi_0(\sigma)
&=
\int_0^1\int_\Omega
\mathcal K_G(\gamma(t)-\sigma(t))
\,d\pi_0(\sigma)\,dt=
\int_0^1\int_{\mathbb R^d}
\mathcal K_G(\gamma(t)-y)
\,d\mu_t(y)\,dt\\
&=
\int_0^1
W_{\pi_0}(t,\gamma(t))\,dt.
\end{align*}
Consequently, the effective path cost takes
the form
\begin{equation}
\label{effective action Gaussian}
c_{\pi_0}(\gamma)
=
\int_0^1
\left[
\frac12|\dot\gamma(t)|^2
-V(t,\gamma(t))
+W_{\pi_0}(t,\gamma(t))
\right]dt.
\end{equation}
The corresponding effective endpoint cost is
\begin{equation}
\label{effective endpoint Gaussian}
\widetilde c_{\pi_0}(x,y)
=
\inf_{\substack{\gamma\in\Omega\\
\gamma(0)=x,\ \gamma(1)=y}}
c_{\pi_0}(\gamma).
\end{equation}

This representation expresses the interaction
through the time marginals of $\pi_0$ and
identifies the effective Lagrangian
\[
L_{\pi_0}(t,x,v)
=
\frac12|v|^2-V(t,x)+W_{\pi_0}(t,x).
\]
It provides the starting point for deriving
the Euler--Lagrange equations and the associated
phase-space dynamics. We subsequently use the
effective endpoint cost to investigate the
optimality structure of the endpoint coupling
and its connection with the eikonal equation.

Now that we have an end-point cost function \eqref{effective action Gaussian} regarding the Gaussian interaction, we define its corresponding \textit{effective cost}, \eqref{effective endpoint Gaussian}. This observation at the level of the optimal path measure allows us to express the interaction through the time marginals \eqref{time marginals optimal plan}.
 From this set up, we derive the corresponding Euler--Lagrange dynamics and show that the associated phase-space marginals satisfy a nonlocal kinetic equation. Then under this knowledge we provide that $W_{\pi_{0}}$ satisfies a viscosity subsolution of \eqref{eqn: eikonal}. Here is a rather general interpretation of Evans argument.


\begin{theorem}
\label{thm:conservative initial momentum}
Let $\eta_0=(\operatorname{id},T^*)_{\#}\mu$
be the endpoint plan associated with $\pi_0$, and suppose that
$\eta_0$ is optimal for the effective endpoint cost
$\widetilde c_{\pi_0}$ \eqref{end-point cost}. Assume that the Evans-type cyclic-shift
argument established above applies on a simply connected open set
$U$, so that
\begin{align}\label{zero circulation theorem Gaussian}
\oint_{\mathcal C}
D_x\widetilde c_{\pi_0}(x,T^*(x))\cdot dx
=
0
\end{align}
for every $C^1$ closed curve $\mathcal C\subset U$.

Assume, in addition, that $\widetilde c_{\pi_0}$ is differentiable \cite{cabrera2021optimal}
in its endpoint variables and that the minimizing trajectory joining
$x$ to $T^*(x)$ is unique and sufficiently regular.
Then there exists a scalar potential
$u^*:U\to\mathbb R$, unique up to an additive constant, such that
\begin{align}\label{Evans potential effective cost}
\nabla u^*(x)
=
D_x\widetilde c_{\pi_0}(x,T^*(x)).
\end{align}
Moreover,
\begin{align}\label{initial momentum identity}
D_x\widetilde c_{\pi_0}(x,T^*(x))
=
-\dot\gamma_x(0),
\end{align}
and therefore
\begin{align}\label{conservative initial velocity}
\nabla u^*(x)
=
-\dot\gamma_x(0).
\end{align}
Thus the initial momentum field of the optimal trajectories is
conservative.
\end{theorem}

\begin{proof}
Equation \eqref{zero circulation theorem Gaussian}, together with
continuity of the vector field
\[
x\longmapsto
D_x\widetilde c_{\pi_0}(x,T^*(x)),
\]
and simple connectedness of $U$, imply by the classical
zero-circulation criterion for conservative vector fields that there exists $u^*$ satisfying
\eqref{Evans potential effective cost}.

It remains to identify the derivative of the endpoint value function.
Let $\gamma=\gamma_{x,y}$ denote the minimizing trajectory joining $x$ to $y$, and, for a Lagrangian stemming from \eqref{effective action Gaussian}-\eqref{effective endpoint Gaussian} $L(t,\gamma(t),\dot{\gamma}(t))$, let
\[
S(x,y)
:=
\widetilde c_{\pi_0}(x,y)
=
\int_0^1
L(t,\gamma(t),\dot\gamma(t))\,dt.
\]
For a smooth family of minimizing trajectories corresponding to
perturbed endpoints, the standard first-variation formula, and identifying $v(t,\gamma):=\dot{\gamma}(t)$, its derivative with respect to $v$ gives
\begin{align*}
D_{\dot{\gamma}(t)}L(t,\gamma(t),\dot{\gamma}(t))&=\frac{d}{d\dot{\gamma}}L(t,\gamma(t),\dot{\gamma(t)})=\frac{d}{d\dot{\gamma}}\left(\tfrac{1}{2}|\dot{\gamma}|^{2}-V(t,\gamma)+W(t,\gamma(t))\right)=\dot{\gamma}.
\end{align*} Therefore,
\begin{align*}
\delta S
&=
\left[
D_vL(t,\gamma(t),\dot\gamma(t))
\cdot\delta\gamma(t)
\right]_{t=0}^{t=1}.
\end{align*}
From which we then obtain,
\[
\delta S
=
\dot\gamma(1)\cdot\delta y
-
\dot\gamma(0)\cdot\delta x.
\]
Consequently, the derivative of $S$, $\delta S$, with respect to $x$ coincides with $D_x\widetilde c_{\pi_0}(x,y)
=
-\dot\gamma_{x,y}(0),$
and the derivative of $S$ with respect to $y$ coincides with $D_y\widetilde c_{\pi_0}(x,y)
=
\dot\gamma_{x,y}(1).$
Setting $y=T^*(x)$ gives
\[
D_x\widetilde c_{\pi_0}(x,T^*(x))
=
-\dot\gamma_x(0).
\]
Combining this identity with
\eqref{Evans potential effective cost} proves
\eqref{conservative initial velocity}.
\end{proof}

\begin{remark}
    It should be well understood that that above differentiation came from the classical minimization of the integral from the calculus of variations, namely the Fundamental Theorem of Calculus. Indeed, letting
    \[
    \mathfrak{I}(\varepsilon):=\int_{x_{1}(\varepsilon)}^{x_{2}(\varepsilon)}f(x,\varepsilon)dx,
    \]
    then its derivative with respect to $\varepsilon$ is
    \[
    \frac{d}{d\varepsilon}\mathfrak{I}(\varepsilon)=f(x_{2},\varepsilon)\frac{dx_{2}}{d\varepsilon}-f(x_{1},\varepsilon)\frac{dx_{1}}{d\varepsilon}+\int_{x_{1}(\varepsilon)}^{x_{2}(\varepsilon)}\frac{\partial f}{\partial \varepsilon}dx,
    \]
    provided that $\partial f/\partial \varepsilon$ is a continuous function of $\varepsilon$ and of $x$ in $x_{1}\leq x \leq x_{2}$.
\end{remark}


\begin{remark}
    Thus $u^*$ and $W_{\pi_0}$ are not, in general, the same
potential. The former generates the conservative initial momentum
field selected by endpoint optimality and cyclical monotonicity,
while the latter generates the self-consistent force acting along
the optimal trajectories.
\end{remark}

Armed with this knowledge we establish a viscosity subsolution to the eikonal equation. But we must first ensure we have enough regularity of $\gamma$ to proceed with this calculation.

\section{From Optimal Path Measures to Phase-Space Dynamics}\label{sec:path-to-phase-space}

We now make precise the passage from the optimal path-space problem
\eqref{eqn: OT paths} to a kinetic equation on phase-space.  There are
three points that require some care. In the first place, regarding regularity of paths $\gamma$, we modify the ambient path space $\Omega$ introduced
earlier in the manuscript, namely the space of absolutely continuous
paths $\gamma:[0,1]\to\mathbb R^d$ to Hilbert space $H^{1}([0,1];\mathbb{R}^{d}):=W^{1,2}([0,1];\mathbb{R}^{d})$. The kinetic cost is extended-valued:
it is finite only when $\dot\gamma\in L^2(0,1;\mathbb R^d)$. Hence any
optimal path measure $\pi_0$ with finite value $\mathcal J(\pi_0)$, defined in \eqref{eq:path-energy-normalized}, is
concentrated on the finite-energy  subset
$W^{1,2}([0,1];\mathbb R^d)$. On this subset, however,
$\dot\gamma(t)$ is still not canonically defined at every fixed time
$t$ until additional regularity is proved. Thus the pointwise velocity
evaluation $\dot e_t(\gamma)=\dot\gamma(t)$ introduced in Section \ref{back ground} will only be applied
after the Euler--Lagrange equation has yielded such regularity.

In the second place, the natural phase-space object generated by the optimal plan is
a probability measure on $\mathbb R^{2d}$ and need not possess a
density with respect to $\mathcal L^{2d}$. In the third place, the sign of the force field must be derived from the precise
Lagrangian defining the path-space problem. We therefore proceed in two stages.  We first establish regularity of
the optimal trajectories from the Euler--Lagrange equation.  We then
define the phase-space lift on those regular trajectories.

\subsection{The effective action and the Gaussian induced force}
\label{subsec:gaussian-effective-action}
Since the single-path kinetic cost is defined to be $+\infty$
whenever $\dot\gamma\notin L^2(0,1;\mathbb R^d)$, any minimizer
$\pi_0$ satisfying $\mathcal J(\pi_0)<\infty$ obeys
\begin{equation}
\label{eq:pi0-finite-energy-concentration}
\pi_0\!\left(W^{1,2}([0,1];\mathbb R^d)\right)=1.
\end{equation}
Indeed, if $\mathfrak{A}:=\Omega\setminus W^{1,2}\left([0, 1];\mathbb{R}^{d}\right)$, then by definition $c(\gamma)=+\infty$ for all $\gamma\in\mathfrak{A}$. We shall show $\pi_{0}\left(\mathfrak{A}\right)=0$, thereby proving $\pi_{0}\left(W^{1,2}([0,1];\mathbb{R}^{d})\right)=1-\pi_{0}(\mathfrak{A})=1-0=1$. To that end, suppose towards sake of a contradiction that $\pi_{0}(\mathfrak{A})>0$. So since $V$ is bounded, there exists $M\geq 0$ such that $V(t,x)\leq M$ for all $(t,x)$ in $[0, 1]\times \mathbb{R}^{d}$. Then as $c$ is bounded from below, $c(\gamma)\geq -M$ for every admissible $\gamma$. As $c(\gamma)=+\infty$ on $\mathfrak{A}$, we have
\begin{align*}
\int_{\Omega} c(\gamma)d\pi_{0}(\gamma)&=\int_{\mathfrak{A}}c(\gamma)d\pi_{0}(\gamma)+\int_{\Omega\setminus\mathfrak{A}}c(\gamma)d\pi_{0}(\gamma)\\
&\geq \int_{\mathfrak{A}}c(\gamma)d\pi_{0}(\gamma)-M\pi_{0}\left(\Omega\setminus\mathfrak{A}\right)\\
&=+\infty.
\end{align*}

Therefore, $\int_{\Omega} c(\gamma) d\pi_{0}(\gamma)=+\infty$. Furthermore, as \; $\mathcal{U}(\gamma,\sigma):=\int_{0}^{1}\lambda e^{-\beta|\gamma(t)-\sigma(t)|^{2}}\;dt\geq 0$ for $\lambda>0$, then this would imply
\[
\mathcal{J}(\pi_{0})=\int_{\Omega}c(\gamma)d\pi_{0}(\gamma)+\tfrac{1}{2}\iint_{\Omega\times\Omega}\mathcal{U}(\gamma, \sigma)d\pi_{0}(\gamma)d\pi_{0}(\sigma)=+\infty,
\]
contradicting the assumption that $\mathcal{J}(\pi_{0})<\infty$.

Thus the Hilbert space,  $W^{1,2}([0,1];\mathbb R^d)$, is the finite-energy subset on which
the optimal measure is concentrated.

We consider the normalized path-space functional
\begin{equation}
\label{eq:path-energy-normalized}
\mathcal J(\pi)
=
\int_{\Omega}\int_0^1
\left[
\frac12|\dot\gamma(t)|^2-V(t,\gamma(t))
\right]dt\,d\pi(\gamma)
+
\frac12
\iint_{\Omega\times\Omega}\int_0^1
\mathcal K_G(\gamma(t)-\sigma(t))
\,dt\,d\pi(\gamma)d\pi(\sigma),
\end{equation}
where $\mathcal{K}_{G}$ is given as in \eqref{eq:gaussian-kernel-dynamics}.
Let $\pi_0\in\Pi_{\mathrm{path}}(\mu,\nu)$  minimize the functional \eqref{eq:path-energy-normalized} and recall \eqref{time marginals optimal plan}.
The self-consistent interaction potential generated by $\pi_0$ is defined by \eqref{self consistent Gaussian potential}, namely, $W_{\pi_0}(t,x)
:=
\int_{\mathbb R^d}\mathcal K_G(x-y)\,d\mu_t(y).$
Since the derivative of the Gaussian is thus
\[
\nabla\mathcal K_G(z)
=
-2\lambda\beta z e^{-\beta|z|^2},
\]
we have
\begin{equation}
\label{eq:grad-W-gaussian}
\nabla_xW_{\pi_0}(t,x)
=
-2\lambda\beta
\int_{\mathbb R^d}(x-y)e^{-\beta|x-y|^2}\,d\mu_t(y).
\end{equation}
The Gaussian interaction term $\mathcal{K}_{G}$ induces its force by
\begin{equation}
\label{eq:exact-force-definition}
K_{\pi_0}(t,x)
:=
-\nabla_xV(t,x)+\nabla_xW_{\pi_0}(t,x).
\end{equation}
The next result quantifies the uniform regularity of the Gaussian interaction functional.
\begin{lemma}
\label{lem:gaussian-force-regularity}
For every $\rho\in\mathcal P(\mathbb R^d)$,  using \eqref{self consistent Gaussian potential} define $W_\rho(x)
:=
(\mathcal K_G*\rho)(x)
=
\lambda\int_{\mathbb R^d}
e^{-\beta|x-y|^2}\,d\rho(y).$

Then
\[
W_\rho\in C^\infty(\mathbb R^d),
\]
and, similarly, by \eqref{eq:grad-W-gaussian}, we get
$\nabla W_\rho(x)
=
-2\lambda\beta
\int_{\mathbb R^d}
(x-y)e^{-\beta|x-y|^2}
\,d\rho(y).$

Moreover,
\begin{align}
\label{Gaussian gradient bound}
\|\nabla W_\rho\|_{L^\infty(\mathbb R^d)}
\leq
\lambda\sqrt{\frac{2\beta}{e}},
\end{align}
and
\begin{align}
\label{Gaussian Hessian bound}
\|D^2W_\rho\|_{L^\infty(\mathbb R^d;\mathrm{op})}
\leq
2\lambda\beta.
\end{align}

Consequently, $\nabla W_\rho$ is globally Lipschitz, with Lipschitz
constant at most $2\lambda\beta$, uniformly in
$\rho\in\mathcal P(\mathbb R^d)$.
\end{lemma}

\begin{proof} 
Every derivative of the Gaussian is a polynomial multiplied by a
Gaussian and is therefore bounded. In particular,
\[
\nabla\mathcal K_G(z)=-2\lambda\beta z e^{-\beta|z|^2}.
\]
Since
\[
\sup_{r\geq0}re^{-\beta r^2}=\frac{1}{\sqrt{2\beta e}},
\]
we obtain
\[
|\nabla\mathcal K_G(z)|
\leq\lambda\sqrt{\frac{2\beta}{e}},
\]
and integration against $\rho$ gives the gradient bound. Furthermore, upon taking a second derivative, we obtain
\begin{align}\label{hessian}
D^2\mathcal K_G(z)
=
\lambda e^{-\beta|z|^2}
\left(4\beta^2z\otimes z-2\beta I\right).
\end{align}
Since $e^{-\beta|z|^2}$ and $|z|^2e^{-\beta|z|^2}$ are bounded,
$D^2\mathcal K_G$ is bounded, and integration against $\rho$ gives
the asserted Hessian bound. Indeed, since the Hessian has the form in \eqref{hessian},
the eigenvalue in the radial direction is $2\lambda\beta e^{-\beta|z|^2}
\left(
2\beta|z|^2-1
\right)$
whereas the eigenvalues in the directions orthogonal to \(z\) are $-2\lambda\beta e^{-\beta|z|^2}.$

Setting \(s=\beta|z|^2\), we get
\[
e^{-s}|2s-1|\le 1,
\qquad s\ge 0.
\]

It follows that
\begin{align*}
\|D^2\mathcal{K}_G(z)\|_{\mathrm{op}}&:=\sup_{|z|=1}|D^{2}\mathcal{K}_{G}(z)| \leq 2\lambda\beta\max\{1,\left|2\beta|z|^{2}-1\right|\}\\
&\le
2\lambda\beta
\qquad
\text{for every } z\in\mathbb{R}^d.
\end{align*}

Since
\[
D^2\mathcal{K}_G(0)
=
-2\lambda\beta I,
\]
the estimate is sharp, and therefore
\[
\|D^2\mathcal{K}_G\|_{L^\infty(\mathbb{R}^d;\mathrm{op})}
=
2\lambda\beta.
\]
\end{proof}
As a consequence, we establish that $W_{\pi_{0}}$ as in \eqref{self consistent Gaussian potential} is a viscosity subsolution of \eqref{eqn: eikonal}.
\begin{corollary}
\label{viscosity solution}
Suppose the hypotheses of Lemma \ref{lem:gaussian-force-regularity}; in addition, that $0<\lambda<\sqrt{\frac{e}{2\beta}}.$
Then, for every $t\in[0,1]$, $W_{\pi_0}(t,\cdot)\in C^\infty(\mathbb R^d)$
and
\begin{align}\label{Gaussian gradient uniform estimate}
\|\nabla_xW_{\pi_0}(t,\cdot)\|_{L^\infty(\mathbb R^d)}
\leq
\lambda\sqrt{\frac{2\beta}{e}}
<1.
\end{align}
Consequently, $W_{\pi_0}(t,\cdot)$ is a classical, and hence
viscosity, subsolution of \eqref{eqn: eikonal}.
\end{corollary}

\begin{proof}
By the previous lemma, we have
\[
\|\nabla_xW_{\pi_0}(t,\cdot)\|_{L^\infty(\mathbb R^d)}
\leq
\lambda\sqrt{\frac{2\beta}{e}}.
\]
Under the assumption $0<\lambda<\sqrt{\frac{e}{2\beta}},$
it follows that

$$
|\nabla_xW_{\pi_0}(t,x)|<1
\qquad
\text{for every }x\in\mathbb R^d.
$$

Thus
$W_{\pi_0}(t,\cdot)$ is a classical subsolution of
\eqref{eqn: eikonal}. Since
$W_{\pi_0}(t,\cdot)\in C^\infty(\mathbb R^d)$, it is, in particular,
a viscosity subsolution.

Finally, since $\mu_0=\mu$,

$$
W_{\pi_0}(0,x)
=
\lambda\int_{\mathbb R^d}
e^{-\beta|x-y|^2}\,d\mu(y)
=
W_\mu(x),
$$

and hence $W_\mu$ is also a viscosity subsolution of
\eqref{eqn: eikonal}.
\end{proof}

\subsection{Gaussian optimality and the Euler--Lagrange equation}
\label{subsec:gaussian-EL}

The next result derives the characteristic equation from optimality of
the path measure itself. We therefore do not assume a priori that an
individual trajectory may be varied while the interaction field is
held fixed. That the path-space optimality implies the Gaussian Euler--Lagrange equation is due to Theorem \ref{thm:gaussian-path-euler-lagrange}. In addition, this result establishes the required regularity of $\gamma$ in the path-space, $\Omega$; which is viewed as $W^{1,2}\left([0, 1]; \mathbb{R}^{d}\right)$.

This result establishes the connection between the path-space optimal transport with congestion and PDE theory. Furthermore, the ideas behind this result originated in the author's dissertation, \cite{Cabrera2022}. Indeed, for any continuously differentiable $u$, formally, minimizers of \eqref{eqn: OT paths} satisfy the nonlocal Euler--Lagrange equation
\begin{align}\label{Geodesic curves}
-\partial_{tt}\Gamma(x,t)
+
\int
u'\bigl(|\Gamma(x,t)-\Gamma(y,t)|\bigr)
\frac{\Gamma(x,t)-\Gamma(y,t)}
{|\Gamma(x,t)-\Gamma(y,t)|}
\,d\mu_0(y)
=
0.
\end{align}
Here $\Gamma: X\to \Omega$ is equivalent to $x\mapsto \gamma_{x}$ \eqref{deterministic plan}-\eqref{conditions of paths}

Thus, the path-space optimal transport problem naturally gives rise to a nonlocal dynamical equation governing the evolution of optimal trajectories. This observation provides the starting point for the kinetic formulation developed in the present work. Equation \eqref{Geodesic curves} provides the formal dynamical motivation for this viewpoint. In the Gaussian case, $u(r)=\lambda e^{-\beta r^2},$
the corresponding nonlocal interaction force takes the form
\[
-2\lambda\beta
\int
\bigl(\Gamma(x,t)-\Gamma(y,t)\bigr)
e^{-\beta|\Gamma(x,t)-\Gamma(y,t)|^2}
\,d\mu_0(y).
\]
In particular, since $W_{\pi_{0}}$ is given by \eqref{self consistent Gaussian potential} then the regularity estimates obtained above allow us to pass from the optimality of the path-space measure to the dynamics of the individual trajectories on which it is concentrated. More precisely, we now show that $\pi_0$-almost every optimal path satisfies the Euler--Lagrange equation associated with the effective Gaussian force. Thus, the variational optimality of $\pi_0$ yields a second-order equation of motion along almost every trajectory. The boundedness of the Gaussian interaction force then gives additional regularity of these paths, which will allow us to introduce their velocities and subsequently pass to the phase-space formulation.

\begin{theorem}
\label{thm:gaussian-path-euler-lagrange}
Assume that $\pi_0$ minimizes the energy functional \eqref{eq:path-energy-normalized}. Suppose
\[
V\in C^1([0,1]\times\mathbb R^d),
\qquad
\|\nabla_xV\|_{L^\infty((0,1)\times\mathbb R^d)}<\infty.
\]
Then, for $\pi_{0}$-almost every $\gamma\in\Omega$,
\begin{equation}
\label{eq:gaussian-EL}
\ddot\gamma(t)
=
K_{\pi_0}(t,\gamma(t))
=
-\nabla_xV(t,\gamma(t))
+
\nabla_xW_{\pi_0}(t,\gamma(t))
\end{equation}
in the distributional sense, i.e., $\mathcal D'((0,1);\mathbb R^d)$. In particular,
\[
\gamma\in W^{2,\infty}(0,1;\mathbb R^d)\subset C^1([0,1];\mathbb R^d)
\]
for $\pi_0$-almost every $\gamma$.
\end{theorem}

\begin{proof}
Fix $h\in C_c^\infty((0,1);\mathbb R^d)$ and let
$a:\Omega\to\mathbb R$ be an arbitrary bounded Borel function. Define the following regularizations via perturbations 
\[
T_\varepsilon(\gamma)(t)
=
\gamma(t)+\varepsilon a(\gamma)h(t),
\qquad
\pi_\varepsilon:=(T_\varepsilon)_{\#}\pi_0.
\]
Since $h(0)=h(1)=0$, both endpoint marginals are preserved, so
$\pi_\varepsilon\in\Pi_{\mathrm{path}}(\mu,\nu)$. Minimality $\pi_{0}$ of $\mathcal{J}$ gives
\[
\left.\frac{d}{d\varepsilon}
\mathcal J(\pi_\varepsilon)\right|_{\varepsilon=0}=0.
\]
The kinetic variation is
\[
\int_\Omega a(\gamma)\int_0^1
\dot\gamma(t)\cdot\dot h(t)\,dt\,d\pi_0(\gamma),
\]
and the external-potential variation is
\[
-\int_\Omega a(\gamma)\int_0^1
\nabla_xV(t,\gamma(t))\cdot h(t)\,dt\,d\pi_0(\gamma).
\]
These differentiations follow respectively from finite kinetic energy
and from boundedness of $\nabla_xV$.

For the interaction term, differentiation under the integral is
justified by the global boundedness of $\nabla\mathcal K_G$. Its
derivative at $\varepsilon=0$ equals
\begin{align*}
\frac12\iint_{\Omega\times\Omega}\int_0^1
\nabla\mathcal K_G(\gamma(t)-\sigma(t))
\cdot[a(\gamma)-a(\sigma)]h(t)
\,dt\,d\pi_0(\gamma)d\pi_0(\sigma).
\end{align*}
Because $\mathcal K_G$ is even, $\nabla\mathcal K_G$ is odd. At any rate,
interchanging $\gamma$ and $\sigma$ in the term containing
$a(\sigma)$ shows that the two contributions coincide. Hence the
factor $1/2$ cancels, and the interaction variation is
\[
\int_\Omega a(\gamma)\int_0^1
\nabla_xW_{\pi_0}(t,\gamma(t))\cdot h(t)\,dt\,d\pi_0(\gamma).
\]
Therefore putting these derivatives together yields
\begin{align}
\int_\Omega a(\gamma)
\left[
\int_0^1\dot\gamma(t)\cdot\dot h(t)\,dt
+
\int_0^1K_{\pi_0}(t,\gamma(t))\cdot h(t)\,dt
\right]d\pi_0(\gamma)=0.
\label{eq:pathwise-localization-pre}
\end{align}
Since $a$ is an arbitrary bounded Borel function, the bracket vanishes
for $\pi_0$-almost every $\gamma$. Thus
\begin{equation}
\label{eq:pathwise-weak-EL}
\int_0^1\dot\gamma(t)\cdot\dot h(t)\,dt
+
\int_0^1K_{\pi_0}(t,\gamma(t))\cdot h(t)\,dt=0.
\end{equation}
Note that $K_{\pi_{0}}$ is given in \eqref{eq:exact-force-definition}.
Using a countable dense family of test functions gives a single
full-measure set on which this identity holds for every
$h\in C_c^\infty((0,1);\mathbb R^d)$. Now the first integral of \eqref{eq:pathwise-weak-EL} yields.
\begin{align*}
    \int_{0}^{1}\dot{\gamma}(t)\cdot\dot{h}(t)dt=-\int_{0}^{1}\ddot{\gamma}(t)\cdot h(t)dt
\end{align*}
by integration by parts. Therefore substituting this into \eqref{eq:pathwise-weak-EL}, we acquire
\begin{align*}
    \int_0^1\left\langle\ddot\gamma(t)
+
K_{\pi_0}(t,\gamma(t)), h(t)\right\rangle\;dt=0\quad\forall t.
\end{align*}
Therefore,
\[
\ddot\gamma(t)=K_{\pi_0}(t,\gamma(t)).
\]
So since $K_{\pi_{0}}$ is given by \eqref{eq:exact-force-definition}, Lemma~\ref{lem:gaussian-force-regularity} shows $W_{\pi_{0}}(t,\cdot)\in C^{\infty}(\mathbb{R}^{d})$ and boundedness of
$\nabla_xV$ imply $K_{\pi_0}\in L^\infty$. Hence
$\ddot\gamma\in L^\infty$, proving
$\gamma\in W^{2,\infty}\subset C^1$.
\end{proof}

\subsection{Phase-Space Marginals and kinetic equation}\label{sec:Phase-Space}
According to Theorem~\ref{thm:gaussian-path-euler-lagrange}, the optimal path measure $\pi_0$ is concentrated on
\[
\Omega^1 := C^1\bigl([0,1];\mathbb{R}^d\bigr).
\]
Indeed, for $\pi_0$-almost every $\gamma$, the Euler--Lagrange equation gives $\ddot{\gamma}(t)
=
K_{\pi_0}\bigl(t,\gamma(t)\bigr)$, as we saw before,
and the boundedness of $K_{\pi_0}$ imply that $\gamma
\in
C^1\bigl([0,1];\mathbb{R}^d\bigr)$.
Then for every $t\in[0,1]$, define the couple
\begin{equation}
\label{eq:position-velocity-evaluation}
E_t:\Omega^1\times T\Omega^{1}\longrightarrow\mathbb R^{2d},
\qquad
E_t(\gamma):=(\gamma(t),\dot\gamma(t)).
\end{equation}
So as $E_t$ is continuous with respect to the $C^1$ topology, it is
Borel measurable.
The next definition supplies the path-space marginal criterion which $\pi_{0}$ helps to generalize.
\begin{definition}
\label{def:phase-space-marginal}
The phase-space measure generated by the optimal path plan $\pi_0$ is
\begin{equation}
\label{eq:phi-phase}
\phi_t:=(E_t)_{\#}\pi_0
=
(\gamma\mapsto(\gamma(t),\dot\gamma(t)))_{\#}\pi_0.
\end{equation}
Its spatial marginal is
\begin{equation}
\label{eq:rho-from-phi}
\mu_t:=(p_x)_{\#}\phi_t=(e_t)_{\#}\pi_0.
\end{equation}
Here  $p_{x}$ is the projection in the first coordinate. Indeed, $p_{x}:\mathbb{R}^{2d}\to \mathbb{R}^{d}$ defined by $p_{x}(x,y)=x$. Now, the composition naturally gives the evaluation map on paths, $p_{x}\circ E_{t}=e_{t}$. Therefore,
\begin{align*}
    (p_{x})_{\#}\phi_{t}&=(p_{x})_{\#}(E_{t})_{\#}\pi_{0}=(p_{x}\circ E_{t})_{\#}\pi_{0}=(e_{t})_{\#}\pi_{0}.
\end{align*}
\end{definition}
In summation, the definition is well-defined.
Lastly, note that no absolute continuity of $\phi_t$ with respect to
$\mathcal L^{2d}$ (Lebesgue measure) is currently assumed.

Recall that $AC(A,B)$ stands for a class of absolutely continuous functions from $A$ to $B$. The next result gives rise to a Gaussian optimal transport which induces a measure-valued kinetic equation. This is the direct connection between optimal transport in the path-space with congestion to PDEs.

\begin{theorem}
\label{thm:path-implies-liouville}
Assume that $\pi_0$ satisfies \eqref{eqn: OT paths} and that it is concentrated on $W^{2,1}$ trajectories
satisfying
\[
\ddot\gamma(t)
=
K_{\pi_0}(t,\gamma(t))
\]
for almost every $t$, and suppose
\begin{equation}
\label{eq:kinetic-integrability}
\int_{\Omega^1}\int_0^1
\left(
|\dot\gamma(t)|
+
|K_{\pi_0}(t,\gamma(t))|
\right)
dt\,d\pi_0(\gamma)
<\infty.
\end{equation}
Then the weakly continuous family $\phi_t=(E_t)_{\#}\pi_0$
satisfies
\begin{equation}
\label{eq:measure-liouville}
\partial_t\phi_t
+
\operatorname{div}_x(v\phi_t)
+
\operatorname{div}_v
\bigl(K_{\pi_0}(t,x)\phi_t\bigr)
=
0
\end{equation}
in $\mathcal D'((0,1)\times\mathbb R^{2d})$.

Equivalently, for every test function
$\psi\in C_c^1((0,1)\times\mathbb R^{2d})$,
\begin{equation}
\label{eq:weak-measure-liouville}
\int_0^1
\int_{\mathbb R^{2d}}
\left[
\partial_t\psi(t,x,v)
+
v\cdot\nabla_x\psi(t,x,v)
+
K_{\pi_0}(t,x)\cdot\nabla_v\psi(t,x,v)
\right]
d\phi_t(x,v)\,dt
=
0,
\end{equation}
which determines weak solutions of \eqref{eq:measure-liouville}.
\end{theorem}

\begin{proof} It is essential for $\pi_{0}$ to attain its infimum \eqref{eqn: OT paths} or equivalently to minimize \eqref{eq:path-energy-normalized} to obtain the Euler-Lagrange equation \eqref{eq:gaussian-EL}. Henceforth, the Euler-Lagrange equation is used to obtain Liouville equations \eqref{PDE Vlasov}-\eqref{PDE Liouville}. To that end, for $\pi_0$-almost every $\gamma$, we have the position of a particle, $\gamma(t)=x$ and its corresponding velocity $v(t,\gamma(t))=\dot{\gamma}(t)$. Now set $Z_\gamma(t):=E_{t}(\gamma)=(\gamma(t),\dot\gamma(t));$ to distinguish that $Z$ is now a function of $\gamma$.
Then $Z_\gamma\in AC([0,1];\mathbb R^{2d})$
and upon differentiating,
\[
\dot Z_\gamma(t)
=
\bigl(
\dot\gamma(t),
K_{\pi_0}(t,\gamma(t))
\bigr)
\]
for almost every $t$.  Therefore an application of the chain-rule yields,
\begin{align*}
\frac{d}{dt}
\psi(t,Z_\gamma(t))=
\partial_t\psi(t,Z_\gamma(t))+
\dot\gamma(t)\cdot
\nabla_x\psi(t,Z_\gamma(t))+K_{\pi_0}(t,\gamma(t))
\cdot
\nabla_v\psi(t,Z_\gamma(t))
\end{align*}
for almost every $t$.
Integrating the preceding identity first with respect to \(t\) and then with respect to \(\pi_0\), and using the integrability assumption \eqref{eq:kinetic-integrability}, we obtain 
\begin{align*}
0 = \int_0^1\int_{\Omega^1} \Big[ \partial_t\psi\bigl(t,\gamma(t),\dot{\gamma}(t)\bigr) &+ \dot{\gamma}(t)\cdot \nabla_x\psi\bigl(t,\gamma(t),\dot{\gamma}(t)\bigr) \\ &+ K_{\pi_0}\bigl(t,\gamma(t)\bigr)\cdot \nabla_v\psi\bigl(t,\gamma(t),\dot{\gamma}(t)\bigr) \Big] \,d\pi_0(\gamma)\,dt. 
\end{align*} 
That the left hand-side of the above equation is identically zero follows from the boundary terms of $\psi$ since they vanish as $\psi$ is compactly supported in \((0,1)\). For each fixed \(t\), the definition \[ \phi_t=(E_t)_{\#}\pi_0, \qquad E_t(\gamma) = \bigl(\gamma(t),\dot{\gamma}(t)\bigr), \] implies that, for every integrable Borel function \(F:\mathbb{R}^{2d}\to\mathbb{R}\), \[ \int_{\mathbb{R}^{2d}} F(x,v)\,d\phi_t(x,v) = \int_{\Omega^1} F\bigl(\gamma(t),\dot{\gamma}(t)\bigr) \,d\pi_0(\gamma). \] Applying this identity with \[ F_t(x,v) = \partial_t\psi(t,x,v) + v\cdot\nabla_x\psi(t,x,v) + K_{\pi_0}(t,x)\cdot\nabla_v\psi(t,x,v), \] we conclude that \[ \int_0^1 \int_{\mathbb{R}^{2d}} \left[ \partial_t\psi(t,x,v) + v\cdot\nabla_x\psi(t,x,v) + K_{\pi_0}(t,x)\cdot\nabla_v\psi(t,x,v) \right] \,d\phi_t(x,v)\,dt = 0, \] which is precisely \eqref{eq:weak-measure-liouville}.
\end{proof}

\subsection{Propagation of phase-space densities and regularity}
The preceding result is naturally formulated at the level of measures and does
not require the phase-space measure \eqref{eq:phase-marginal-from-path-measure} to be absolutely continuous with respect to Lebesgue measure on
\(\mathbb R^{2d}\). In fact, such absolute continuity need not follow from the
path-space formulation alone. It is therefore useful to distinguish the
measure-valued kinetic equation from the additional situation in which the
initial phase-space measure possesses a density.

Next, we show that the characteristic flow propagates, for the Gaussian interaction, absolute continuity. More precisely, if $\phi_0=f_0\,\mathcal L^{2d},$ then $\phi_t=f_t\,\mathcal L^{2d}$ for every $t\in[0,1]$. Moreover, the divergence-free structure of the
phase-space vector field implies preservation of the $L^p$-norms of the
density.

\begin{theorem}
\label{thm:gaussian-density-propagation}
Let $\mathcal K_G(z)=\lambda e^{-\beta |z|^2},
\; \lambda,\beta>0,$
and let $K_{\pi_0}(t,x)
=
-\nabla_xV(t,x)
+
\nabla_xW_{\pi_0}(t,x),$
where $W_{\pi_0}(t,x)$ is given by \eqref{self consistent Gaussian potential}. Assume that $\nabla_xV\in L^\infty([0,1]\times\mathbb R^d)$ and that there exists \(L_V>0\) such that
\[
|\nabla_xV(t,x)-\nabla_xV(t,y)|
\le L_V|x-y|
\]
for every $x,y\in\mathbb R^d$ and almost every $t\in [0, 1]$. Suppose that the phase-space measures $\phi_{t}$, given by \eqref{eq:phase-marginal-from-path-measure},
satisfy
\[
\partial_t\phi_t
+
\operatorname{div}_x(v\phi_t)
+
\operatorname{div}_v
\bigl(K_{\pi_0}(t,x)\phi_t\bigr)
=0
\]
in the sense of distributions.

If the initial phase-space measure is absolutely continuous,
\[
\phi_0=f_0\,\mathcal L^{2d},
\qquad
f_0\in L^1(\mathbb R^{2d}),
\]
then for every \(t\in[0,1]\), $\phi_t=f_t\,\mathcal L^{2d}$ for some $f_t\in L^1(\mathbb R^{2d})$. Moreover, \(f=f(x,v,t)\) is the unique distributional solution of
\[
\partial_t f
+
v\cdot\nabla_x f
+
K_{\pi_0}(t,x)\cdot\nabla_v f
=0,
\]
with initial data $f(\cdot,\cdot,0)=f_0.$

More precisely, if

\begin{align*}
    \Psi_t(x,v)
=
\bigl(X_t(x,v),V_t(x,v)\bigr)
\end{align*}
denotes the characteristic flow satisfying $\dot X_t=V_t,
\quad
\dot V_t=K_{\pi_0}(t,X_t),$ then $\phi_t=(\Psi_t)_{\#}\phi_0$ and $f_t=f_0\circ\Psi_t^{-1}
\qquad\text{for a.e. }(x,v)\in\mathbb R^{2d}.$
In particular,
\[
\|f_t\|_{L^p(\mathbb R^{2d})}
=
\|f_0\|_{L^p(\mathbb R^{2d})},
\qquad 1\le p\le\infty,\quad\text{whenever}\quad f_0\in L^p(\mathbb R^{2d}).
\]

\end{theorem}

\begin{proof}
Since we have $
\nabla\mathcal K_G(z)
=
-2\lambda\beta\,z\,e^{-\beta|z|^2},
$ then automatically we get
\[
\nabla_xW_{\pi_0}(t,x)
=
-2\lambda\beta
\int_{\mathbb R^d}
(x-y)e^{-\beta|x-y|^2}\,d\mu_t(y).
\]
Furthermore to this,
\[
D^2\mathcal K_G(z)
=
-2\lambda\beta e^{-\beta|z|^2}I
+
4\lambda\beta^2e^{-\beta|z|^2}z\otimes z.
\]
Since $\sup_{z\in\mathbb R^d}|D^2\mathcal K_G(z)|<\infty$, thanks to Lemma \ref{lem:gaussian-force-regularity} there exists a constant \(C_{\lambda,\beta}<\infty\), depending only on
\(\lambda\) and \(\beta\), such that
\[
\|D^2W_{\pi_0}(t,\cdot)\|_{L^\infty}
\le C_{\lambda,\beta}
\]
for every \(t\in[0,1]\). Therefore

$$
|\nabla W_{\pi_0}(t,x)-\nabla W_{\pi_0}(t,y)|
\le
C_{\lambda,\beta}|x-y|.
$$

Using $K_{\pi_{0}}(t,\cdot)$ as given in \eqref{eq:exact-force-definition} and combining the latter estimate with the Lipschitz assumption on $\nabla V$,  we obtain through the triangle inequality,
\begin{align*}
|K_{\pi_0}(t,x)-K_{\pi_0}(t,y)|&=|-\nabla_{x}V(t,x)+\nabla_{x}W_{\pi_{0}}(t,x)+\nabla_{y}V(t,y)-\nabla_{y}W_{\pi_{0}}(t,y)|\\
&\leq|\nabla_{x}V(t,x)-\nabla_{y}V(t,y)|+|\nabla_{x}W_{\pi_{0}}(t,x)-\nabla_{y}W_{\pi_{0}}(t,y)|\\
&\leq(L_V+C_{\lambda,\beta})|x-y|.
\end{align*}
Thus \(K_{\pi_0}(t,\cdot)\) is globally Lipschitz uniformly in \(t\).

Define the phase-space vector field $B: \mathbb{R}^{2d}\to \mathbb{R}^{2d}$ by 
\[B(t,x,v):=\left(v, K_{\pi_{0}}(t,x)\right).\] 
Then
\[
|B(t,x,v)-B(t,y,w)|
\le
|v-w|
+
(L_V+C_{\lambda,\beta})|x-y|.
\]
Hence \(B(t,\cdot,\cdot)\) is globally Lipschitz on phase-space, uniformly in \(t\). 
The phase-space vector field, $B(t,x,v)$, is divergence-free. Indeed,
\[
\operatorname{div}_{(x,v)}B(t,x,v)
=
\operatorname{div}_x v
+
\operatorname{div}_vK_{\pi_0}(t,x).
\]
Since $v$ is independent of $x$, $\operatorname{div}_xv=0,$
and since $K_{\pi_0}(t,x)$ is independent of $v$, $\operatorname{div}_vK_{\pi_0}(t,x)=0.$
Consequently, the vector field $B(t,x,v)$ is incompressible
\[
\operatorname{div}_{(x,v)}B=0.
\]
In particular, 
\begin{align*}
    \text{div}_{(x,v)}\left(f(x,v,t)(v,K_{\pi_{0}}(t,x))\right)&=\nabla_{(x,v)}f(x,v,t)\cdot (v, K_{\pi_{0}}(t,x))\\
    &+f(x,v,t)\cdot \text{div}_{(x,v)}(v, K_{\pi_{0}}(t,x))\\
    &=v\cdot \nabla_{x}f(x,v,t)+K_{\pi_{0}}(t,x)\cdot\nabla_{v}f(x,v,t).
\end{align*}
THerefore, the kinetic equation 
\[
\partial_{t}f+v\cdot \nabla_{x}f+K_{\pi_{0}}(t,x)\cdot\nabla_{v}f=0
\]
is equivalent to the continuity equation
\[
\partial_{t}f(t,x,v)+\text{div}_{(x,v)}\left(f(t,x,v)(v,K_{\pi_{0}}(t,x))\right)=0.
\]

Returning to the fact that the vector field, $B(t, \cdot,\cdot)$ is globally Lipschitz on phase-space and uniformly int $t$, we make the above rigorous with the aid of $\Psi_{t}(x,v)=(X_{t}(x,v), V_{t}(x,v))$. The characteristic system
\[
\dot X_t=V_t,
\qquad
\dot V_t=K_{\pi_0}(t,X_t)
\]
therefore induces a unique flow
\[
\Psi_t:\mathbb R^{2d}\longrightarrow\mathbb R^{2d}.
\]
Thus, the vector field can be expressed in the form $B(t,x,v):=B_{t}(z)$, where $z=(x,v)\in \mathbb{R}^{2d}$. Then
\[
\frac{d}{dt}\Psi_{t}(z)=B(z, \Psi_{t}(z)),\quad \Psi_{0}(z)=z.
\]

First, we want to establish that the determinant $\det D\Psi_{t}(z)=1$. Note that $D\Psi_{t}(z)$ is the $2d\times 2d$ Jacobian matrix of the flow with respect to the initial condition $z=(x,v)$. Next, we deduce Lioville's formula. 

Define $A(t,z)=D_{z}\Psi_{t}(z)$. Differentiating the characteristic equation with respect to the initial condition $z$; hence denoted by differentiation with respect to $z$, $D_{z}$, we obtain
\begin{align*}
\frac{d}{dt}A(t,z)=\frac{d}{dt}\left(D_{z}\Psi_{t}(z)\right)
&=D_{z}\left(\frac{d}{dt}\Psi_{t}(z)\right)\\
&=D_{z}\left(B(t, \Psi_{t}(z))\right).
\end{align*}
The latter is where we apply the chain-rule with respect to $z$ coming from differentiation of $D_{z}$. Indeed, 
\[
D_{z}\left(B(t, \Psi_{t}(z))\right)=D_{z}B(t,\Psi_{t}(z))D_{z}\Psi_{t}(z)=D_{z}B(t,\Psi_{t}(z))A(t,z).
\]
Thereofre, 
\begin{align}\label{jac formula}
    \frac{d}{dt}A(t,z)=D_{z}B(t, \Psi_{t}(z))A(t,z).
\end{align}

Now recalling Jacobian's formula for the derivative of a determinant, we have
\[
\frac{d}{dt}\det(A(t,z))=\text{tr}\left(\text{adj}\left(A(t,z)\right)A^{\prime}(t,z)\right).
\]
Whenever, $A(t,z)$ is invertible, this can be written as 
\[
\frac{d}{dt}\det A(t,z)=\det A(t,z)\text{tr}\left(A^{-1}(t,z)A^{\prime}(t,z)\right).
\]
Substituting \eqref{jac formula} into the above $A^{\prime}$, yields,
\[
\frac{d}{dt}\det(A(t,z))=\det A(t,z)\text{tr}\left(A^{-1}(t,z)D_{z}B(t,\Psi_{t}(z))A(t,z)\right).
\]
Applying the cyclic invariance of the trace in the above equality yields,
\[
\frac{d}{dt}\det(A(t,z))=\det A(t,z)\text{tr}\left(D_{z}B(t,\Psi_{t}(z))\right)
\]
But the trace of the Jacobian matrix is precisely the divergence, 
\[
\text{tr}D_{z}B=\text{div}B,
\]
so that then
\[
\frac{d}{dt}A(t,z)=\det A(t,z)\text{div}\left(B(t,\Psi_{t}(z))\right)
\]

Recalling that $A(t,z)=D_{z}\Psi_{t}(z)$, the Jacobian of the characteristic flow therefore satisfies Liouville's formula \cite{ArnoldODE}
\[
\frac{d}{dt}
\det D\Psi_t(z)
=
(\operatorname{div}B)(t,\Psi_t(z))
\det D\Psi_t(z)
=0
\]
as $\text{div}B=0$. In particular, since $
B(t,x,v)=(v,K_{\pi_0}(t,x))$, its Jacobian with respect to the phase-space variables $(x,v)$ is
\[
D_{(x,v)}B(t,x,v)
=
\begin{pmatrix}
0 & I_d\\
D_xK_{\pi_0}(t,x) & 0
\end{pmatrix}.
\]

Consequently, the variational equation becomes
\[
\frac{d}{dt}A(t,z)
=
\begin{pmatrix}
0 & I_d\\
D_xK_{\pi_0}(t,X_t(z)) & 0
\end{pmatrix}
A(t,z),
\qquad
A(0,z)=I_{2d}.
\]

Therefore, $d/dt(\det D\Psi_{t}(z))=0$ so that by the fundamental theorem of calculus, that is integrating the latter on $[0, s]$, we acquire $\det D\Psi_{s}(z)=\det D\Psi_{0}(z)$. But initially the flow is the identity map, $\Psi_{0}(z)=z$ so that therefore $D\Psi_{0}(z)=I_{2d}$, the identity $2d\times 2d$ matrix. So since $\det D\Psi_0(z)=1,$ we obtain $\det D\Psi_t(z)=1.$ Thus the phase-space flow preserves Lebesgue measure.

As the continuity equation associated with the globally Lipschitz field \(B\)
has a unique measure solution, one has the appropriate measure $\phi_t=(\Psi_t)_{\#}\phi_0.$
Using $\phi_0=f_0\,\mathcal L^{2d},$ for every bounded Borel function $h$, the change of variables formula yields,
\[
\int_{\mathbb R^{2d}}h(z)\,d\phi_t(z)
=
\int_{\mathbb R^{2d}}
h(\Psi_t(z))f_0(z)\,dz.
\] Making the change of variables $y=\Psi_t(z)$, and assuming $\Psi_{t}(z)$ is a diffeomorphism and using the Jacobian factor $d^{d}y=|\det D\Psi_{t}(z)|d^{d}z$, from the change of variables $y=\Psi_{t}(z)$ with the arbitrariness of the composition $h\circ \Psi_{t}(z)$,  we obtain $f_{0}(z)=|\det D\Psi_{t}(z)|\phi_{0}(\Psi_{t}(z))$ \cite{McCannGuillen2013}. Now, we have
\[
\int_{\mathbb R^{2d}}h(y)\,d\phi_t(y)
=
\int_{\mathbb R^{2d}}
h(y)f_0(\Psi_t^{-1}(y))\,dy.
\]
Since $\det D\Psi_t=1,$ we acquire $f_{0}(z)=\phi_{0}(\Psi_{t}(z))$. Therefore,
\[
\phi_t=f_t\,\mathcal L^{2d},
\qquad
f_t(y)=f_0(\Psi_t^{-1}(y)).
\]
Hence $f$ satisfies the PDE
\[
\partial_t f+
\operatorname{div}_{(x,v)}(Bf)=0
\]
in the distributional sense. Since $\operatorname{div}_{(x,v)}B=0,$ this is equivalent to
\[
\partial_t f
+
v\cdot\nabla_x f
+
K_{\pi_0}(t,x)\cdot\nabla_v f
=0.
\]
This indicates that $B$ is divergence free, and the above equation stems from 
\[
\partial_{t}f+B\cdot\nabla_{x,v}f=0.
\]
This is an equivalent construction from \cite{ambrosio2014lagrangian} where the authors observe that the Vlasov-Poisson system has a transport strucuture which is given exactly by the above equation.

Finally, since \(\Psi_t\) preserves Lebesgue measure, we have the following $L^{p}$-norm equality
\[
\|f_t\|_{L^p}^{p}
=
\int_{\mathbb R^{2d}}
|f_0(\Psi_t^{-1}(z))|^p\,dz
=
\int_{\mathbb R^{2d}}
|f_0(z)|^p\,dz=\|f_{0}\|_{L^{p}}^{p}
\]
for $1\leq p<\infty$, and the same change-of-variables argument gives $\|f_t\|_{L^\infty}
=
\|f_0\|_{L^\infty}.$
This proves the theorem.
\end{proof}

In particular, this result offers a partial answer to the following question: Given $\phi_{t}$ as Theorem \ref{thm:gaussian-density-propagation}, when can we say there is a vector field $K:\mathbb{R}^{n}\to \mathbb{R}^{n}$  and a function $f(x,v,t)$ such that 
\[
\phi_{t}=f(\cdot, \cdot, t)dxdv\quad\forall t
\]
and $f$ solves (in some weak sense )
\begin{align*}
\partial_{t}f(x,v,t)+v\cdot\nabla_{x}f(x,v,t)+K(x)\cdot\nabla_{v}f(x,v,t)=0.
\end{align*}
However, what this theorem answers is the density/PDE aspect of the question more directly, with one important caveat: it merely applies to the time-dependent self-consistent force $K_{\pi_0}(t,x)$; not some general $K(x)$.

In particular, for the Gaussian kernel, the force in \eqref{eq:measure-liouville} is given by \eqref{eq:gaussian-self-consistent-force} with $\mu_t=(p_x)_{\#}\phi_t$.
Consequently we have the system,
\begin{equation}
\label{eq:gaussian-kinetic-system}
\begin{cases}
\displaystyle
\partial_t\phi_t
+
\operatorname{div}_x(v\phi_t)
+
\operatorname{div}_v\!\left(K_{\pi_0}(t,x)\phi_t\right)=0,
\\[2mm]
\displaystyle
K_{\pi_0}(t,x)
=
-\nabla_xV(t,x)
-
2\lambda\beta
\int_{\mathbb R^d}
(x-y)e^{-\beta|x-y|^2}\,d\mu_t(y),
\\[3mm]
\displaystyle
\mu_t=(p_x)_{\#}\phi_t.
\end{cases}
\end{equation}
Thus the Gaussian path-space minimizer generates a measure-valued
solution of a nonlinear, nonlocal kinetic equation.

\begin{remark}
\label{rem:no-phase-density}
Theorem \ref{thm:path-implies-liouville} is deliberately stated for
probability measures.  Even if the initial spatial measure is
absolutely continuous, a deterministic family $x\longmapsto
(\gamma_x(t),\dot\gamma_x(t))$
maps a $d$-dimensional parameter space into the $2d$-dimensional
phase-space.  Its image measure may therefore be singular, that is $\mu$ is concentrated on a set to which $\nu$ assigns zero measure, with respect
to $\mathcal L^{2d}$.  Consequently, the representation, a differential measure with respect to Lebesgue measure, $d\phi_t(x,v)=f(t,x,v)\,dx\,dv$
must not be assumed without an additional hypothesis.
\end{remark}

\subsection{The converse lifting and superposition principle}

We now formulate a partial converse statement to Theorem \ref{thm:gaussian-density-propagation} in the appropriate phase-space
setting. At each time $t$, a probability measure on trajectories induces
a distribution of particles by evaluation at time $t$. Namely, starting from a path measure $\pi$, not necessarily optimal, the phace-space distribution at time $t$ is obtained through prjection of path-space $\phi_{t}=(e_{t},\dot{e}_{t})_{\#}\pi$. Conversely, given a dynamical family of probability measures in the phase-space, $\{\phi_{t}\}_{t\in[0,1]}$, we seek to reconstruct a probability measure on trajectories whose time marginals match with $\phi_{t}$. 

It is important to distinguish the measure arising in this superposition
principle from the endpoint coupling introduced previously. The end-point coupling introduced previously is $\eta_{0}:=(e_{0}, e_{1})_{\#}\pi_{0}=(\text{id}, T^{*})_{\#}\mu\in\mathcal{P}(\mathbb{R}^{2d})$, and records only the initial and final positions of the trajectories under the optimal path measure $\pi_{0}$. By contrast, the superposition principle of the dynamics of curves in euclidean space, produces a measure
 $\eta_{\mathrm{ph}}$. Namely,
\[
\eta_{\mathrm{ph}}
\in
\mathcal P\bigl(AC([0,1];\mathbb R^{2d})\bigr)
\]
is a probability measure on phase-space trajectories $Z(\cdot)=(X(\cdot),V(\cdot)).$
So $\eta_{0}$ is the end-point coupling, while $\eta_{\textbf{ph}}$ is a probability measure on position-velocity trajectories.

For each $t\in [0,1 ]$, the phase-space superposition principle, define the evaluation map
\[
\operatorname{ev}_t:
AC([0,1];\mathbb R^{2d})
\longrightarrow
\mathbb R^{2d}
\]
by
\[
\operatorname{ev}_t(Z)
=
Z(t)
=
\bigl(X(t),V(t)\bigr).
\]
Then the identity
\begin{align}\label{eq:phase-superposition-marginal}
(\operatorname{ev}_t)_{\#}\eta_{\mathrm{ph}}
=
\phi_t
\end{align}
means that the time-$t$ marginal of the probability measure
$\eta_{\mathrm{ph}}$ on phase-space trajectories is precisely the
phase-space distribution $\phi_t$. In other words, the above measure identiy \emph{is} the time-$t$ marginal of the superposition measure which precisely presrcibes pahse-space distribution. After estabilishing the required regularity of the charactersistcs, we project $\eta_{\textbf{ph}}$ onto its spatial component by setting $P(X, V)$ and $\pi:=P_{\#}\eta_{\textbf{ph}}$.

The resulting path measure $\pi$ is not asserted to be optimal; rather, it realizes the prescribed phase-space evolution in the sense that $(e_{t}, \dot{e}_{t})_{\#}\pi=\phi_{t}$ for each $t\in [0, 1]$.

We now pass from the Eulerian phase-space description to a Lagrangian representation by characteristics. Let \[ t\longmapsto \phi_t\in\mathcal P(\mathbb R^{2d}) \] be weakly continuous, and let \[ K:(0,1)\times\mathbb R^d\longrightarrow\mathbb R^d \] be Borel measurable. We associate to $K$ the phase-space vector field \begin{equation} \label{eq:phase-vector-field} B(t,x,v):=(v,K(t,x)). \end{equation} We assume the integrability condition, almost identical as \eqref{eq:kinetic-integrability} but not quite, 
\begin{equation} \label{eq:superposition-integrability} \int_0^1\int_{\mathbb R^{2d}} \bigl(|v|+|K(t,x)|\bigr) \,d\phi_t(x,v)\,dt<\infty, \end{equation} 
and that $\phi$ solves the phase-space continuity equation \begin{equation} \label{eq:phase-continuity-assumed} \partial_t\phi_t + \operatorname{div}_{(x,v)} \bigl(B(t,x,v)\phi_t\bigr) =0 \end{equation} in the sense of distributions. The characteristics associated with \eqref{eq:phase-vector-field} are phase-space curves $Z(t)=(X(t),V(t))$ satisfying 
\begin{equation} \label{eq:characteristic-system} \dot X(t)=V(t), \qquad \dot V(t)=K(t,X(t)) \quad\text{for a.e. }t\in(0,1). \end{equation} 
The superposition principle applied to \eqref{eq:phase-continuity-assumed} now yields the following Lagrangian representation. That is, the following result provides the corresponding lifting of the evolving
phase-space distributions $\{\phi_t\}_{t\in[0,1]}$ to a probability measure
on characteristic trajectories.

\begin{theorem}
\label{thm:phase-space-superposition}
Let $t\mapsto\phi_t\in\mathcal P(\mathbb R^{2d})$ be weakly
continuous and suppose that $\phi$ satisfies
\eqref{eq:phase-continuity-assumed} with the phase-space vector field
\eqref{eq:phase-vector-field}, under the integrability condition
\eqref{eq:superposition-integrability}.
Then there exists
\[
\eta_{\mathrm{ph}}
\in\mathcal P\bigl(AC([0,1];\mathbb R^{2d})\bigr)
\]
such that
\[
(\operatorname{ev}_t)_\#\eta_{\mathrm{ph}}=\phi_t,
\qquad t\in[0,1],
\]
and $\eta_{\mathrm{ph}}$ is concentrated on solutions of
\eqref{eq:characteristic-system}. In particular,
\begin{equation}
\label{eq:spatial-marginal-projection}
(e_t)_\#\pi=(p_x)_\#\phi_t.
\end{equation}
\end{theorem}

\begin{proof}
Apply the superposition principle for the continuity equation on
$\mathbb R^{2d}$ to the weakly continuous solution
$\{\phi_t\}_{t\in[0,1]}$ of
\eqref{eq:phase-continuity-assumed}, with the Borel vector field
$B$ defined by \eqref{eq:phase-vector-field}. To that emd, for every bounded function $f\in C_{b}(\mathbb{R}^{2d})$, we have $t\mapsto \int f(x,v)d\phi_{t}(x,v)$ is continuous. Now since $\phi_{t}$ solves \eqref{eq:phase-continuity-assumed}, for every test function $f\in C^{1}_{c}((0,1);\mathbb{R}^{2d})$, we have, by integration by parts,
\begin{align*}
0&=\int_{0}^{1}\int_{\mathbb{R}^{2d}}\left(\partial_{t}\phi_{t}f+\text{div}_{(x,v)}(B\phi)f\right)dzdt =-\int_{0}^{1}\int_{\mathbb{R}^{2d}}\left(\partial_{t}f+B\cdot\nabla_{z}f\right)d\phi_{t}(z)dt.
\end{align*}

So since, $B(t,x,v)=(v, K(t,x))$, then 
\begin{align*}
    |B(t,x,v)|=|(v,K(t,x))|&=\sqrt{|v|^{2}+|K(t,x)|^{2}}\\
    &\leq |v|+|K,x(t)|.
\end{align*}
Thus according to
\eqref{eq:superposition-integrability} gives precisely the required
space-time integrability condition,
\[
\int_0^1
\int_{\mathbb R^{2d}}
|B(t,x,v)|
\,d\phi_t(x,v)\,dt
<\infty,
\]
up to the equivalence of norms on $\mathbb R^{2d}$.

The superposition principle therefore yields a probability measure
\[
\eta_{\mathrm{ph}}
\in
\mathcal P\bigl(AC([0,1];\mathbb R^{2d})\bigr)
\]
such that
\[
(\operatorname{ev}_t)_{\#}\eta_{\mathrm{ph}}
=
\phi_t
\]
for every $t\in[0,1]$, and such that
$\eta_{\mathrm{ph}}$ is concentrated on integral curves of $B$.
Writing such a curve as $Z=(X,V)$
we obtain $\dot Z(t)
=
B(t,Z(t))
=
\bigl(V(t),K(t,X(t))\bigr)$
for almost every $t$. Equivalently,
\[
\dot X(t)=V(t),
\qquad
\dot V(t)=K(t,X(t))
\]
for almost every $t$.

We next verify the additional regularity. From
\eqref{eq:phase-superposition-marginal} and Tonelli's theorem, we have 
\begin{align*}
\int_{AC([0,1];\mathbb R^{2d})}
\left(
\int_0^1 |K(t,X(t))|\,dt
\right)
d\eta_{\mathrm{ph}}(X,V)
 =
\int_0^1
\int_{\mathbb R^{2d}}
|K(t,x)|
\,d\phi_t(x,v)\,dt
<\infty.
\end{align*}
Hence, $K(\cdot,X(\cdot))\in L^1(0,1;\mathbb R^d)$
for $\eta_{\mathrm{ph}}$-almost every trajectory. So since $\dot V(t)=K(t,X(t))$
almost everywhere, it follows that $V\in AC([0,1];\mathbb R^d)$.
In particular, $V$ is continuous. Moreover, the fundamental theorem of calculus applies to show
\begin{align*}
    \int_{0}^{t}V(s)ds=\int_{0}^{t}\dot{X}(s)ds&=X(t)-X(0)\\
    X(0)+\int_{0}^{t}V(s)ds&=X(t),
\end{align*}
and therefore $X\in C^1([0,1];\mathbb R^d)$ with $\dot X(t)=V(t)$
for every $t\in[0,1]$.

Now define the following path measure, $\pi$, not necessarily optimal, by
\[
\pi:=P_{\#}\eta_{\mathrm{ph}},
\qquad
P(X,V)=X.
\]
For every bounded function $\zeta:\mathbb R^{2d}\longrightarrow\mathbb R,$ we have
\begin{align*}
\int_{\mathbb R^{2d}}
\zeta(x,v)\,
d\bigl((e_t,\dot e_t)_{\#}\pi\bigr)(x,v)
&=
\int
\zeta\bigl(X(t),\dot X(t)\bigr)
\,d\eta_{\mathrm{ph}}(X,V)
\\
&=
\int
\zeta\bigl(X(t),V(t)\bigr)
\,d\eta_{\mathrm{ph}}(X,V)
\\
&=
\int_{\mathbb R^{2d}}
\zeta(x,v)\,d\phi_t(x,v),
\end{align*}
where the last equality follows from $(\operatorname{ev}_t)_{\#}\eta_{\mathrm{ph}}
=
\phi_t.$
Hence $(e_t,\dot e_t)_{\#}\pi
=
\phi_t.$

As a consequence, we have the recovery of the path measure. That is projecting onto the spatial coordinate gives
\[
(e_t)_{\#}\pi
=
(p_x)_{\#}
\bigl((e_t,\dot e_t)_{\#}\pi\bigr)
=
(p_x)_{\#}\phi_t,
\]
which proves \eqref{eq:spatial-marginal-projection} below.
\end{proof}

\begin{remark}
\label{rem:answer-lifting}
Theorem \ref{thm:phase-space-superposition} gives the appropriate
converse statement to Theorem \ref{thm:gaussian-density-propagation}.  An arbitrary collection
$\{\phi_t\}$ of one-time probability measures is not enough.
A sufficient dynamical compatibility condition is that the family be
weakly continuous and solve a phase-space continuity equation whose
first component is exactly the velocity variable $v$.  The
superposition principle then reconstructs a probability measure on
position--velocity characteristics and hence, after projection, a
probability measure on $C^1$ position paths.
\end{remark}
\section{Acknowledgments} The author is deeply grateful to his advisor, Nestor Guillen, for introducing this problem and for insightful discussions that took place at The University of Texas at Austin during Spring 2023 and Fall 2023. 

The author also acknowledges the use of \textbf{artificial intelligence} tools such as ChatGPT and Copilot in the preparation of this manuscript: These tools assisted in generating Figure 1, improving the organization of the paper, and helping with literature search and producing exploratory computations. 

\section*{Data Availability}

Data sharing is not applicable to this article, as no
datasets were generated or analyzed during the current study.

\bibliography{OTCF}
\bibliographystyle{plain}

\end{document}